\documentclass[final,onefignum,onetabnum]{siamart171218}

\usepackage{lipsum}
\usepackage{amsfonts}
\usepackage{graphicx}
\usepackage{epstopdf}
\usepackage{algorithmic}
\ifpdf
  \DeclareGraphicsExtensions{.eps,.pdf,.png,.jpg}
\else
  \DeclareGraphicsExtensions{.eps}
\fi

\newsiamremark{remark}{Remark}
\newsiamremark{hypothesis}{Hypothesis}
\crefname{hypothesis}{Hypothesis}{Hypotheses}
\newsiamthm{claim}{Claim}

\newsiamthm{example}{Example} 

\usepackage{multirow}

\headers{A Randomized Coordinate Descent Method with Volume Sampling}{A. Rodomanov and D. Kropotov}

\title{A Randomized Coordinate Descent Method with Volume Sampling\thanks{Submitted to the editors DATE.
\funding{This research is in part based on the work
supported by Samsung Research, Samsung Electronics. Results in
Sections~\ref{section:convergence_results:convex_functions},~\ref{section:convergence_results:strongly_convex_functions} have been obtained by Dmitry Kropotov and supported by the Russian Science Foundation grant no.∼19-71-30020.}}}

\author{Anton Rodomanov\thanks{Samsung-HSE
Laboratory, National Research University Higher School of
Economics, Moscow, Russia 
  (\email{anton.rodomanov@gmail.com}, 
  \email{dmitry.kropotov@gmail.com}).}
\and Dmitry Kropotov\footnotemark[2]
\thanks{Lomonosov Moscow State University, Moscow, Russia.}}

\usepackage{amsopn}

\usepackage{graphicx}

\usepackage{mathtools}
\newcommand{\R}{\mathbb{R}} % real numbers
\renewcommand{\Pr}{\mathbb{P}} % probability
\newcommand{\E}{\mathbb{E}} % expectation
\DeclareMathOperator*{\Argmin}{Argmin} % The set of minimizers
\DeclareMathOperator*{\Det}{Det} % determinant
\DeclareMathOperator*{\Rank}{Rank} % rank
\DeclareMathOperator*{\Ker}{Ker} % kernel
\DeclareMathOperator*{\Img}{Im} % image
\DeclareMathOperator{\Tr}{Tr} % trace
\DeclareMathOperator*{\Diag}{Diag} % diagonal matrix
\DeclareMathOperator*{\Adj}{Adj} % adjugate matrix
\DeclareMathOperator*{\Vol}{Vol} % volume
\DeclareMathOperator*{\RCDVS}{RCDVS} % RCDVS (short name for the proposed method)
\DeclareMathOperator*{\sparsetwovspreprocess}{sparse2vs\_preprocess}
\DeclareMathOperator*{\sparsetwovssample}{sparse2vs\_sample}
\DeclareMathOperator{\nnz}{nnz} % nnz (number of non-zeros)

\begin{document}

\maketitle

\begin{abstract}
We analyze the coordinate descent method with a new coordinate selection strategy, called volume sampling. This strategy prescribes selecting subsets of variables of certain size proportionally to the determinants of principal submatrices of the matrix, that bounds the curvature of the objective function. In the particular case, when the size of the subsets equals one, volume sampling coincides with the well-known strategy of sampling coordinates proportionally to their Lipschitz constants. For the coordinate descent with volume sampling, we establish the convergence rates both for convex and strongly convex problems. Our theoretical results show that, by increasing the size of the subsets, it is possible to accelerate the method up to the factor which depends on the spectral gap between the corresponding largest eigenvalues of the curvature matrix. Several numerical experiments confirm our theoretical conclusions.
\end{abstract}



\section{Introduction}

Coordinate descent methods are minimization algorithms that are very popular for solving large-scale optimization problems. The main idea of these algorithms is to successively reduce the value of the objective function along certain subsets of coordinates that are selected at each iteration according to some rule. Coordinate descent has been successfully applied to a number of applications in various areas such as machine learning, compressed sensing, network problems etc.

One of the main parameters of a typical coordinate descent method is the \emph{number of coordinates} $\tau$, which are updated at every iteration. When choosing this parameter, one usually faces the following trade-off. On the one hand, the convergence rate of the method becomes faster with the increase of $\tau$, but, on the other hand, each iteration becomes more expensive. Therefore, to obtain a speed-up of the method in terms of the total running time, it is necessary to ensure that the increase in the cost of each iteration is low, compared to the increase in the convergence rate. One possible way to achieve this is \emph{parallelization}. This idea has been extensively discussed in the context of parallel coordinate descent~\cite{richtarik2016parallel,necoara2017random,fercoq2016optimization,necoara2016parallel,bradley2011parallel,peng2013parallel}, where (under certain separability assumptions) the authors show that the convergence rate improves \emph{linearly} in $\tau$, and thus it is possible to achieve a linear speed-up by using $\tau$ independent processors instead of one. Note, however, that in the usual \emph{serial regime} (without parallelization) the aforementioned results do not guarantee any decrease in the total running time, since each iteration becomes at least $\tau$ times more expensive. Clearly, to be able to ensure a speed-up in this regime, one needs some \emph{non-linear}, in terms of $\tau$, convergence rate estimates.

In this paper, we analyze the coordinate descent method with a new coordinate selection strategy, called \emph{volume sampling}. This strategy prescribes selecting $\tau$-element subsets of variables proportionally to the volumes (or determinants) of principal submatrices of the matrix $B$, that bounds the curvature of the objective function (see Section~\ref{section:description_of_method}). For the coordinate descent with volume sampling, we establish the worst-case iteration complexity bounds, that have a \emph{non-linear} dependency on $\tau$. In particular, we show that the increase in $\tau$ from $\tau_1$ to $\tau_2$ leads to the improvement of the convergence rate of the method up to the factor of
\begin{equation}\label{eq:theoretical_acceleration_ratio}
R_{\lambda}(\tau_1, \tau_2) := \frac{\sum_{i=\tau_1}^n \lambda_i}{\sum_{i=\tau_2}^n \lambda_i},
\end{equation}
where $\lambda_1 \geq \dots \geq \lambda_n$ are the eigenvalues of $B$. Note that $R_{\lambda}(\tau_1, \tau_2)$ can be arbitrarily big, depending on the \emph{spectral gap} between $\lambda_{\tau_1}$ and $\lambda_{\tau_2}$.

In addition to this, we also propose a new efficient algorithm for 2-element volume sampling from a \emph{sparse} matrix. The preprocessing complexity of this algorithm is of the order of the number of non-zero elements in the matrix, and its sampling complexity is only logarithmic in the dimension.

\subsection{Related work}\label{section:related_work}

There is a vast literature on coordinate descent methods. Most research in this area is usually focused on the rules for selecting coordinates \cite{nesterov2012efficiency,beck2013convergence,nutini2015coordinate}, or on obtaining accelerated \cite{nesterov2012efficiency,nesterov2017efficiency,allen2016even}, parallel \cite{richtarik2016parallel,necoara2017random,fercoq2016optimization,necoara2016parallel,bradley2011parallel,peng2013parallel}, proximal \cite{fercoq2016optimization,richtarik2014iteration} and primal-dual \cite{shalev2013stochastic,lin2015accelerated,shalev2016sdca,qu2016sdna} variants of the already known methods. For a general overview of the topic, see the recent paper~\cite{wright2015coordinate} and references therein. Below we discuss just several works that are most closely related to ours.

One of the most influential papers on coordinate descent is \cite{nesterov2012efficiency} by Nesterov, where he proposed a coordinate gradient method (which we will refer to as RCD) with a special \emph{randomized} rule for selecting coordinates. In RCD, each coordinate is sampled with probability proportional to the corresponding coordinate Lipschitz constant. Nesterov then derived the complexity bound for RCD and showed that it can be even better than that of the standard gradient method. The method, that we consider in this work, generalizes RCD in the sense that it coincides with it in the special case when the number of coordinates, selected at each iteration, equals one.

The most relevant work to ours is~\cite{qu2016sdna}, where the authors propose three different randomized methods for smooth unconstrained minimization. Their Method~1 is exactly the same method, that we analyze in this paper, with the only difference that, instead of volume sampling, they consider an \emph{arbitrary} sampling. As a result, the convergence rate of their method is expressed quite abstractly (in terms of the minimal eigenvalue of the expectation of a certain matrix), and, in particular, it is not clear how exactly it depends on the number of coordinates $\tau$, used at each iteration. Although the authors provide a particular example of a $3 \times 3$ matrix, for which their method should be very efficient, they do not establish any general results. In this regard, our work can be considered a further development of \cite{qu2016sdna}, where we establish more interpretable iteration complexity bounds specifically for volume sampling. In addition to that, in \cite{qu2016sdna}, the authors work only with the strongly convex problems, while. in this paper, we allow the problem to be non-strongly convex.

Another relevant work to this one is \cite{kovalev2018stochastic}, where the authors propose a new randomized optimization method for minimizing quadratic functions, given $\tau$ eigenvectors corresponding to the $\tau$ smallest eigenvalues of the Hessian. Although the complexity estimates for this method look similar to ours, there are several key differences. First, the results in \cite{kovalev2018stochastic} show that the increase in the number of coordinates from $\tau_1$ to $\tau_2$ in their method leads to the acceleration rate that depends on the spectral gap between the $\tau_1$st and $\tau_2$nd \emph{smallest} eigenvalues. The acceleration rate for coordinate descent with volume sampling, depends, on the contrary, on the spectral gap between the $\tau_1$st and $\tau_2$nd \emph{largest} eigenvalues. Second, their method is not, strictly speaking, a coordinate descent algorithm since it uses eigenvectors as search directions instead of the coordinate directions. Finally, it is also less practical. For example, even in the simplest non-trivial regime $\tau=1$, their method requires an eigenvector, corresponding to the smallest eigenvalue; the complexity of obtaining such a vector is in general $O(n^3)$. In contrast, the simplest non-trivial choice for coordinate descent with volume sampling is $\tau=2$, which requires to perform at the beginning one preprocessing step of $O(n^2)$ (or even $O(n+\nnz(B))$, see Section~\ref{section:sparse_two_element_volume_sampling}), and then each iteration of the method takes linear time in the dimension.

Finally, we should mention that, although volume sampling has not been previously considered in the context of coordinate descent methods, it is not a novel concept, and has already been known in the literature for some time. To our knowledge, it was first proposed in \cite{deshpande2006matrix} for the problem of matrix approximation. Later on the same authors developed several efficient exact and approximate methods for doing volume sampling \cite{deshpande2010efficient} based on the standard linear algebra algorithms. Some other polynomial-time sampling methods and their connection to the theory of Markov chains were considered in \cite{li2017polynomial}. Recently volume sampling has also been applied to the problem of linear regression \cite{derezinski2018reverse,derezinski2018leveraged}.

\subsection{Contents}

This paper is organized as follows. In Section~\ref{section:description_of_method}, we describe the randomized coordinate descent method with volume sampling. In Section~\ref{section:convergence_results}, we present the convergence analysis of this method. We start with an auxiliary sufficient decrease lemma (Section~\ref{section:sufficient_decrease_lemma}) and then use it to derive the convergence rates both for convex functions (Section~\ref{section:convergence_results:convex_functions}) and strongly convex ones (Section~\ref{section:convergence_results:strongly_convex_functions}).  In Section~\ref{section:implementation_of_volume_sampling}, we discuss how to generate a random variable according to volume sampling. First, we discuss a simple general approach (Section~\ref{section:general_algorithm}) and, after this, develop a special algorithm for 2-element volume sampling which is suitable for sparse matrices (Section~\ref{section:sparse_two_element_volume_sampling}). In Section~\ref{section:examples_of_applications}, we consider several examples of possible applications: quadratic functions (Section~\ref{section:quadratic_function}), separable problems (Section~\ref{section:separable_problems}) and the smoothing technique (Section~\ref{section:smoothing_technique}). Finally, in Section~\ref{section:numerical_experiments}, we present the results of several numerical experiments.

\subsection{Notation}

By $\R^n$ we denote the Euclidean space of all $n$-dimensional real column vectors with the standard inner product $\langle u, v \rangle := \sum_{i=1}^n u_i v_i$ and the standard Euclidean norm $\|v\| := \langle v, v \rangle^{\frac12}$. Given an $n \times n$ real symmetric positive semidefinite matrix $B$, we also use the seminorm $\|v\|_B := \langle B v, v \rangle^{\frac12}$; recall that $\| \cdot \|_B$ becomes a norm iff $B$ is positive definite.

For $1 \leq \tau \leq n$, by $[n] \choose \tau$ we denote the collection of all $\tau$-element subsets of $[n] := \{1,\dots,n\}$. For each $S \in {[n] \choose \tau}$, by $I_S$ we denote the $n \times \tau$ matrix obtained from the $n \times n$ identity matrix $I$ by retaining only those columns whose indices are in $S$; if the dimension $n$ is not specified directly, then it can be determined from the context. For an $n \times n$ matrix $B$ and a subset $S \in {[n] \choose \tau}$, by $B_{S \times S}$ we denote the $\tau \times \tau$ principal submatrix located at the intersection of the rows and columns with indices from $S$ (i.e. $B_{S \times S} := I_S^T B I_S$); similarly, for a vector $v \in \R^n$, by $v_S$ we denote the subvector of size $\tau$ obtained from $v$ by retaining only the elements with indices from $S$ (i.e. $v_S := I_S^T v$).

Finally, for a square matrix $A$, by $\Adj(A)$ we denote its adjugate matrix (the transpose of the cofactor matrix).

\section{Randomized coordinate descent with volume sampling}\label{section:description_of_method}

Consider the unconstrained optimization problem
$$
\min_{x \in \R^n} f(x),
$$
where $f : \R^n \to \R$ is a differentiable function. We assume that $f$ is \emph{1-smooth with respect to the seminorm $\| \cdot \|_B$} induced by some $n \times n$ real symmetric positive semidefinite matrix $B$:
\begin{equation}\label{eq:smoothness_wrt_B}
f(y) \leq f(x) + \langle \nabla f(x), y - x \rangle + \frac{1}{2} \| y - x \|_B^2
\end{equation}
for all $x, y \in \R^n$ (see Section~\ref{section:examples_of_applications} for examples). When $f$ is twice continuously differentiable, one sufficient condition for this is that the Hessian of $f$ is uniformly upper bounded by $B$.

\begin{remark}
The standard smoothness assumption in the context of coordinate descent methods is slightly different. Typically, it is assumed that, for each $S \in {[n] \choose \tau}$, there exists $L_S \geq 0$ (called coordinate Lipschitz constant) such that
\begin{equation}\label{eq-std-lip}
f(x + I_S h) \leq f(x) + \langle \nabla f(x)_S, h \rangle + \frac{L_S}{2} \| h \|^2
\end{equation}
for all $x \in \R^n$ and all $h \in \R^\tau$. Clearly, if $f$ satisfies \eqref{eq:smoothness_wrt_B}, it also satisfies \eqref{eq-std-lip} for $L_S := \| B_{S \times S} \|$. However, \eqref{eq:smoothness_wrt_B} and \eqref{eq-std-lip} are not completely equivalent. For example, if $\tau=1$ and the function $f$ is twice continuously differentiable, \eqref{eq-std-lip} requires only the \emph{diagonal} of the Hessian to be uniformly bounded. Nevertheless, for many practical applications, condition \eqref{eq:smoothness_wrt_B} holds (see Section~\ref{section:examples_of_applications}).
\end{remark}

Let us fix a point $x_0 \in \R^n$ and a $\tau$-element subset of coordinates $S_0 \in {[n] \choose \tau}$, where $1 \leq \tau \leq \Rank(B)$. According to \eqref{eq:smoothness_wrt_B}, we have
\begin{equation}\label{eq:smoothness_wrt_B_coordinate_version}
f(x_0 + I_{S_0} h) \leq f(x_0) + \langle \nabla f(x_0)_{S_0}, h \rangle + \frac{1}{2} \| h \|^2_{B_{S_0 \times S_0}}
\end{equation}
for all $h \in \R^\tau$. A natural idea to obtain an update rule of a coordinate descent algorithm is to minimize the right-hand side of \eqref{eq:smoothness_wrt_B_coordinate_version} in $h$. It is possible to do so when the matrix $B_{S_0 \times S_0}$ is non-degenerate, and this leads to the following update rule:
\begin{equation}\label{eq:basic_iteration}
x_1 := x_0 - I_{S_0} (B_{S_0 \times S_0})^{-1} \nabla f(x_0)_{S_0}.
\end{equation}

Now it remains to specify the procedure for selecting the coordinates $S_0$. In view of the above remark, the probability of choosing a degenerate submatrix $B_{S_0 \times S_0}$ should be zero. One sampling scheme that naturally possesses this property is given by the following

\begin{definition}[Volume sampling]
Let $B$ be an $n \times n$ real symmetric positive semidefinite matrix, let $1 \leq \tau \leq \Rank(B)$, and let $S_0$ be a random variable taking values in $[n] \choose \tau$. We say that $S_0$ is generated according to \emph{$\tau$-element volume sampling with respect to $B$}, denoted by $S_0 \sim \Vol_{\tau}(B)$, if for all $S \in {[n] \choose \tau}$, we have
$$
\Pr(S_0 = S) = \frac{\Det(B_{S \times S})}{\sum_{S' \in {[n] \choose \tau}} \Det(B_{S' \times S'})}.
$$
\end{definition}
Observe that for $\tau=1$ volume sampling corresponds to picking indices with probabilities proportional to the coordinate Lipschitz constants $B_{ii}$ (diagonal elements of $B$). Thus, volume sampling in fact generalizes the well-known coordinate Lipschitz constant sampling. We discuss its implementation in Section~\ref{section:implementation_of_volume_sampling}.

Combining the update rule \eqref{eq:basic_iteration} with the volume sampling of coordinates, we obtain the Randomized Coordinate Descent Method with Volume Sampling (RCDVS), see Algorithm~\ref{algorithm:rcdvs}. Note that for $\tau=1$ RCDVS coincides with the well-known RCD method from \cite{nesterov2012efficiency}.

\begin{algorithm}
\caption{\label{algorithm:rcdvs}$\RCDVS(f, B, \tau, x_0, K)$}
\begin{algorithmic}
\REQUIRE differentiable function $f : \R^n \to \R$; $n \times n$ real symmetric positive semidefinite matrix $B$; number of coordinates $1 \leq \tau \leq \Rank(B)$; starting point $x_0 \in \R^n$; number of iterations $K \geq 1$.
\FOR{$0 \leq k \leq K-1$}
    \STATE Choose a random subset of coordinates $S_k \sim \Vol_{\tau}(B)$.
    \STATE Set $x_{k+1} := x_k - I_{S_k} (B_{S_k \times S_k})^{-1} \nabla f(x_k)_{S_k}$.
\ENDFOR
\end{algorithmic}
\end{algorithm}

\section{Convergence analysis}\label{section:convergence_results}

We now turn to analyzing the convergence rate of the RCDVS method. To keep the presentation concise, we only study the convergence rates of expectations, although it is not difficult to establish their high probability counterparts using standard techniques.

\subsection{Sufficient decrease lemma}\label{section:sufficient_decrease_lemma}

We start with the following simple result which directly follows from smoothness and does not yet take into account the particular strategy for sampling coordinates:
\begin{lemma}[General sufficient decrease lemma]\label{lemma:sufficient_decrease_lemma_1}
Let $f : \R^n \to \R$ be a 1-smooth function with respect to the seminorm $\| \cdot \|_B$, where $B$ is an $n \times n$ real symmetric positive semidefinite matrix. Let $x_0 \in \R^n$ be deterministic, let $1 \leq \tau \leq \Rank(B)$, let $S_0$ be a random variable taking values in ${[n] \choose \tau}$ such that $B_{S_0 \times S_0}$ is non-degenerate almost surely, and let $x_1 := x_0 - I_{S_0} (B_{S_0 \times S_0})^{-1} \nabla f(x_0)_{S_0}$. Then $f(x_1) \leq f(x_0)$ almost surely, and
\begin{equation}\label{eq:sufficient_decrease_lemma_1}
f(x_0) - \E f(x_1) \geq \frac{1}{2} \| \nabla f(x_0) \|_{\E I_{S_0} (B_{S_0 \times S_0})^{-1} I_{S_0}^T}^2.
\end{equation}
\end{lemma}

In its current form, Lemma~\ref{lemma:sufficient_decrease_lemma_1} is not very useful since it involves some general expectation $\E I_{S_0} (B_{S_0 \times S_0})^{-1} I_{S_0}^T$ which is not clear how to work with. Our task now is to estimate this expectation in a convenient yet non-trivial way for the particular case $S_0 \sim \Vol_\tau(B)$.

Assume that all $\tau \times \tau$ submatrices of $B$ are non-degenerate, i.e. the $\tau$-element volume sampling has full support (the other case will be considered later). Using Cramer's rule $\Det(B_{S \times S}) (B_{S \times S})^{-1} = \Adj(B_{S \times S})$, we can write
\begin{equation}\label{eq:expectation_via_sum_simple_case}
\E I_{S_0} (B_{S_0 \times S_0})^{-1} I_{S_0}^T = \frac{ \sum_{S \in {[n] \choose \tau}} I_S \Adj(B_{S \times S}) I_S^T }{ \sum_{S \in {[n] \choose \tau}} \Det(B_{S \times S}) }.
\end{equation}
Thus, to estimate the expectation, we need to estimate the following two sums:
\begin{enumerate}
\item The sum of principal minors $\sum_{S \in {[n] \choose \tau}} \Det(B_{S \times S})$.
\item The sum $\sum_{S \in {[n] \choose \tau}} I_S \Adj(B_{S \times S}) I_S^T$.
\end{enumerate}

The first sum is rather well-known and a closed form expression for it can be found in many standard textbooks on linear algebra (see e.g. Chapter 7 \cite{meyer2000matrix}). To present the formula, let us introduce for each $1 \leq m \leq n$, the real \emph{elementary symmetric polynomial} $\sigma_m : \R^n \to \R$ of degree $m$, defined by
$$
\sigma_m(x) := \sum_{ 1 \leq i_1 < \dots < i_m \leq n } x_{i_1} \dots x_{i_m},
$$
i.e. the sum of all $m$-ary products of $x_1, \dots, x_n$, and put $\sigma_0(x) := 1$ for convenience. The well-known result is
\begin{lemma}[Sum of principal minors]\label{lemma:sum_of_principal_minors}
Let $B$ be an $n \times n$ real symmetric matrix with eigenvalues $\lambda := (\lambda_1, \dots, \lambda_n)$, where $\lambda_1 \geq \dots \geq \lambda_n$, and let $1 \leq \tau \leq n$. Then
\begin{equation}\label{lemma:sum_of_principal_minors:main_eq}
\sum_{S \in {[n] \choose \tau}} \Det(B_{S \times S}) = \sigma_{\tau}(\lambda).
\end{equation}
\end{lemma}

Now we turn to the second sum. To the best of our knowledge, this sum has not been previously considered in the literature. Nevertheless, it turns out that it can also be conveniently expressed in terms of the elementary symmetric polynomials of eigenvalues:
\begin{lemma}\label{lemma:sum_of_adjugates_of_principal_submatrices}
Let $B$ be an $n \times n$ real symmetric matrix with eigenvalues $\lambda := (\lambda_1, \dots, \lambda_n)$, where $\lambda_1 \geq \dots \geq \lambda_n$, let $B = Q \Diag(\lambda) Q^T$ be its spectral decomposition for some $n \times n$ orthogonal matrix $Q$, and let $1 \leq \tau \leq n$. Then
\begin{equation}\label{eq:sum_of_adjugates_of_principal_submatrices}
\sum_{S \in {[n] \choose \tau}} I_S \Adj(B_{S \times S}) I_S^T = Q \Diag(\sigma_{\tau-1}(\lambda_{-1}), \dots, \sigma_{\tau-1}(\lambda_{-n})) Q^T,
\end{equation}
where $\lambda_{-i}$ for each $1 \leq i \leq n$ denotes the vector $\lambda$ without the $i$-th element.
\end{lemma}
Let us accept this lemma for now and defer its proof to a separate Section~\ref{section:proof_of_key_estimate}.

Using Lemma~\ref{lemma:sum_of_principal_minors} together with Lemma~\ref{lemma:sum_of_adjugates_of_principal_submatrices}, we can rewrite \eqref{eq:expectation_via_sum_simple_case} as follows:
\begin{equation}\label{eq:expectation_via_sum_simple_case_2}
\E I_{S_0} (B_{S_0 \times S_0})^{-1} I_{S_0}^T = \frac{ Q \Diag(\sigma_{\tau-1}(\lambda_{-1}), \dots, \sigma_{\tau-1}(\lambda_{-n})) Q^T }{ \sigma_{\tau}(\lambda) }.
\end{equation}
Thus, we have managed to express the expectation solely in terms of the eigenvectors and eigenvalues of $B$. However, our new expression for the expectation is still difficult to work with because each elementary symmetric polynomial is in fact a very complex sum. Fortunately, recall that we do not need the expectation itself but only a suitable lower bound for it. To obtain such a bound, it is convenient to introduce
\begin{definition}[$\tau$-coordinate approximation]
Let $B$ an $n \times n$ real symmetric positive semidefinite matrix with eigenvalues $\lambda_1 \geq \dots \geq \lambda_n$, let $1 \leq \tau \leq \Rank(B)$, and let $B = Q \Diag(\lambda_1, \dots, \lambda_n) Q^T$ be a spectral decomposition of $B$ for some $n \times n$ orthogonal matrix $Q$. The \emph{$\tau$-coordinate approximation of $B$}, denoted by $B_\tau$, is the $n \times n$ real positive semidefinite matrix
\begin{equation}\label{def-b-tau}
B_\tau := Q \Diag(\lambda_1, \dots, \lambda_\tau, \lambda_\tau, \dots, \lambda_\tau) Q^T + \sum_{i=\tau+1}^n \lambda_i I.
\end{equation}
\end{definition}

Observe that $B_\tau$ is non-degenerate since otherwise $\lambda_\tau = \dots = \lambda_n = 0$ which, in view of positive semidefiniteness, contradicts the assumption that $\tau \leq \Rank(B)$. Also note that $B_\tau$ does not depend on the particular orthogonal matrix $Q$ in the spectral decomposition of $B$. Indeed, the first term in the definition of $B_\tau$ can be written as $Q \Diag(q_\tau(\lambda_1), \dots, q_\tau(\lambda_n)) Q^T$, where $q_\tau : \R \to \R$ is the function $q_\tau(t) := \max\{ t, \lambda_\tau \}$. It is well-known that such matrices do not depend on the choice of the diagonalizing matrix $Q$ (see e.g. \cite[Section 7.3]{meyer2000matrix}).

Using the $\tau$-coordinate approximation, we can now lower bound \eqref{eq:expectation_via_sum_simple_case_2} as follows:
\begin{lemma}
Let $B$ an $n \times n$ real symmetric positive semidefinite matrix with eigenvalues $\lambda_1 \geq \dots \geq \lambda_n$, let $1 \leq \tau \leq \Rank(B)$, and let $B = Q \Diag(\lambda_1, \dots, \lambda_n) Q^T$ be a spectral decomposition of $B$ for some $n \times n$ orthogonal matrix $Q$. Then
$$
\frac{ Q \Diag(\sigma_{\tau-1}(\lambda_{-1}), \dots, \sigma_{\tau-1}(\lambda_{-n})) Q^T}{ \sigma_{\tau}(\lambda) } \succeq (B_\tau)^{-1}.
$$
\end{lemma}

\begin{proof}
Since the eigenvalues are non-negative, we have
$$
\sigma_{\tau}(\lambda) = \sum_{i_1=1}^{n-\tau+1} \lambda_{i_1} \sum_{i_1+1 \leq i_2 < \dots < i_\tau \leq n} \lambda_{i_2} \dots \lambda_{i_\tau}  \leq \sigma_{\tau-1}(\lambda_{-1}) \sum_{i=1}^{n-\tau+1} \lambda_i.
$$
By the symmetry of elementary symmetric polynomials, this can strengthened to
$$
\sigma_\tau(\lambda) \leq \sigma_{\tau-1}(\lambda_{-1}) \left( \lambda_1 + \sum_{j=\tau+1}^n \lambda_j \right),
$$
which in turn can be further generalized to
$$
\sigma_\tau(\lambda) \leq \sigma_{\tau-1}(\lambda_{-i}) \left( \lambda_{\min\{i, \tau\}} + \sum_{j=\tau+1}^n \lambda_j \right)
$$
for all $1 \leq i \leq n$. The claim follows.
\end{proof}

To summarize, we have obtained that in the case when volume sampling has full support, we can replace \eqref{eq:sufficient_decrease_lemma_1} with
\begin{equation}\label{eq:sufficient_decrease_2_aux}
f(x_0) - \E f(x_1) \geq \frac{1}{2} \| \nabla f(x_0) \|_{(B_\tau)^{-1}}^2.
\end{equation}
It remains to show that exactly the same result holds even when some $\tau \times \tau$ principal submatrices of $B$ are possibly degenerate. In this case, instead of \eqref{eq:expectation_via_sum_simple_case} we should write more carefully that
$$
\E I_{S_0} (B_{S_0 \times S_0})^{-1} I_{S_0}^T = \frac{ \sum_{S \in {[n] \choose \tau} : \Det(B_{S \times S}) \neq 0} I_S \Adj(B_{S \times S}) I_S^T }{\sum_{S \in {[n] \choose \tau}} \Det(B_{S \times S})}.
$$

Unfortunately, we cannot use Lemma~\ref{lemma:sum_of_adjugates_of_principal_submatrices} anymore because now \eqref{eq:sum_of_adjugates_of_principal_submatrices} \emph{overestimates} (and not underestimates) the numerator due to the fact that the adjugate to a symmetric positive semidefinite matrix is also symmetric positive semidefinite. However, recall that we are not interested in the numerator itself, but only in how it acts on the gradient $\nabla f(x_0)$. For each $1 \leq \tau \leq n$, define the linear subspace
$$
U_\tau(B) := \{ u \in \R^n : \text{$u_S \in \Img(B_{S \times S})$ for all $S \in {\textstyle{[n] \choose \tau}}$ with $\Det(B_{S \times S}) = 0$} \},
$$
where $\Img(B_{S \times S})$ is the image space of $B_{S \times S}$. Observe that $U_\tau(B) = \R^n$ if and only if all $\tau \times \tau$ principal submatrices of $B$ are non-degenerate. Our interest in the subspace $U_\tau(B)$ lies in the following observation:
\begin{lemma}\label{lemma:position_of_the_gradient}
Let $f : \R^n \to \R$ be a 1-smooth function with respect to the seminorm $\| \cdot \|_B$, where $B$ is an $n \times n$ real symmetric positive semidefinite matrix. If $f$ is bounded from below, then for each $x_0 \in \R^n$ and each $1 \leq \tau \leq n$, we have $\nabla f(x_0) \in U_\tau(B)$.
\end{lemma}

\begin{proof}
Let $S \in {[n] \choose \tau}$ be such that $\Det(B_{S \times S}) = 0$ (if there is no such $S$, the claim is vacuously true). Then the kernel $\Ker(B_{S \times S})$ is non-trivial and hence there exists a non-zero $h \in \Ker(B_{S \times S})$. From smoothness of $f$ with respect to $\| \cdot \|_B$, it follows that $f(x_0 + t I_S h) \leq f(x_0) + t \langle \nabla f(x_0)_S, h \rangle$ for all $t \in \R$. Hence, $\nabla f(x_0)_S \in \Ker(B_{S \times S})^\perp = \Img(B_{S \times S})$, otherwise $f$ is unbounded from below.
\end{proof}

According to Lemma~\ref{lemma:position_of_the_gradient} and the above remarks, we are interested only in the action of $\sum_{S \in {[n] \choose \tau} : \Det(B_{S \times S}) \neq 0} I_S \Adj(B_{S \times S}) I_S^T$ on the subspace $U_\tau(B)$. But one can easily see that on this subspace it acts exactly as the already studied matrix $\sum_{S \in {[n] \choose \tau}} I_S \Adj(B_{S \times S}) I_S^T$, and so the case of degenerate submatrices reduces to that of non-degenerate ones.

Thus, regardless of whether there are degenerate principal submatrices or not, we have proved
\begin{lemma}[Sufficient decrease lemma for volume sampling]\label{lemma:sufficient_decrease_lemma_2}
Let $f : \R^n \to \R$ be a function which is bounded from below and 1-smooth with respect to the seminorm $\| \cdot \|_B$, where $B$ is an $n \times n$ real symmetric positive semidefinite matrix. Let $x_0 \in \R^n$, let $S_0 \sim \Vol_{\tau}(B)$ for some $1 \leq \tau \leq \Rank(B)$, and let $x_1 := x_0 - I_{S_0} (B_{S_0 \times S_0})^{-1} \nabla f(x_0)_{S_0}$. Then $f(x_1) \leq f(x_0)$ almost surely and
\begin{equation}
f(x_0) - \E f(x_1) \geq \frac{1}{2} \| \nabla f(x_0) \|_{(B_\tau)^{-1}}^2.
\end{equation}
\end{lemma}

To finish this section, let us establish the following relations between $\tau$-coordinate approximations that will be central in the forthcoming convergence analysis.

\begin{lemma}\label{lemma:relations_between_different_tau}
Let $B$ be an $n \times n$ real symmetric positive semidefinite matrix with eigenvalues $\lambda_1 \geq \dots \geq \lambda_n$, and let $1 \leq \tau_1 < \tau_2 \leq \Rank(B)$. Then
$$
B_{\tau_2} \preceq B_{\tau_1} \preceq R_{\lambda}(\tau_1, \tau_2) B_{\tau_2},
$$
where $R_{\lambda}(\tau_1, \tau_2)$ is defined by \eqref{eq:theoretical_acceleration_ratio}.
\end{lemma}

\begin{proof}
Denote $\Sigma_i := \sum_{j=i}^n \lambda_j$ for $1 \leq i \leq n$. According to \eqref{def-b-tau}, we have $B_{\tau_1} = Q \Diag(s^{(1)}_1, \ldots, s^{(1)}_n) Q^T$, where
$$
s^{(1)}_i :=
\begin{cases}
\lambda_i + \Sigma_{\tau_1+1}, & \text{if $1 \leq i \leq \tau_1$}, \\
\Sigma_{\tau_1}, & \text{if $\tau_1+1 \leq i \leq n$}.
\end{cases}
$$
Similarly, $B_{\tau_2} = Q \Diag(s^{(2)}_1, \ldots, s^{(2)}_n) Q^T$, where
$$
s^{(2)}_i :=
\begin{cases}
\lambda_i + \Sigma_{\tau_2+1}, & \text{if $1 \leq i \leq \tau_2$}, \\
\Sigma_{\tau_2}, & \text{if $\tau_2+1 \leq i \leq n$}.
\end{cases}
$$
Hence, $(B_{\tau_2})^{-1/2} B_{\tau_1} (B_{\tau_2})^{-1/2} = Q \Diag(s_1, \ldots, s_n) Q^T$, where
$$
s_i :=
\begin{cases}
\frac{\lambda_i + \Sigma_{\tau_1+1}}{\lambda_i + \Sigma_{\tau_2+1}}, & \text{if $1 \leq i \leq \tau_1$}, \\
\frac{\Sigma_{\tau_1}}{\lambda_i + \Sigma_{\tau_2+1}}, & \text{if $\tau_1+1 \leq i \leq \tau_2$}, \\
\frac{\Sigma_{\tau_1}}{\Sigma_{\tau_2}}, & \text{if $\tau_2+1 \leq i \leq n$}.
\end{cases}
$$
To finish the proof, it now remains to show that $s_1, \ldots, s_n$ are bounded from below by 1 and bounded from above by $\frac{\Sigma_{\tau_1}}{\Sigma_{\tau_2}}$.

For $1 \leq i \leq \tau_1$, we have $s_i = 1 + \frac{\Sigma_{\tau_1+1} - \Sigma_{\tau_2+1}}{\lambda_i + \Sigma_{\tau_2+1}}$. Since $\Sigma_{\tau_1+1} \geq \Sigma_{\tau_2+1} \geq 0$, it follows from this representation that $1 \leq s_1 \leq \ldots \leq s_{\tau_1}$ (recall that $\lambda_1 \geq \ldots \geq \lambda_n \geq 0$). Similarly, for $\tau_1 \leq i \leq \tau_2$ (including $\tau_1$), we have $s_i = \frac{\Sigma_{\tau_1}}{\lambda_i + \Sigma_{\tau_2+1}}$, hence $s_{\tau_1} \leq \ldots \leq s_{\tau_2}$. Finally, $s_{\tau_2} = \ldots = s_n$ (including $\tau_2$). Thus, $1 \leq s_1 \leq \ldots \leq s_n = \frac{\Sigma_{\tau_1}}{\Sigma_{\tau_2}}$.
\end{proof}

\subsection{Convex functions}\label{section:convergence_results:convex_functions}

Now we are ready to establish several results on the convergence rate of the RCDVS method. We start with the class of smooth convex functions.

\begin{theorem}[Convergence rate for convex functions]\label{theorem:convergence_for_convex_functions}
Let $f : \R^n \to \R$ be a convex function which is 1-smooth with respect to the seminorm $\| \cdot \|_B$, where $B$ is an $n \times n$ real symmetric positive semidefinite matrix. Let $x_0$ be a deterministic point in $\R^n$ and assume that the sublevel set $L_f(x_0) := \{ x \in \R^n : f(x) \leq f(x_0) \}$ is bounded. Let $1 \leq \tau \leq \Rank(B)$, $K \geq 1$, and let $(x_k)_{k=1}^K$ be the random points in $\R^n$ generated by $\RCDVS(f, B, \tau, x_0, K)$. Then
$$
\E f(x_k) - \min f \leq \frac{2 D_\tau^2}{k+1},
$$
for all $0 \leq k \leq K$, where $D_\tau := \max_{x \in L_f(x_0)} \min_{x^* \in \Argmin f} \| x - x^* \|_{B_{\tau}}$ is the radius of the sublevel set $L_f(x_0)$ measured in the norm $\| \cdot \|_{B_{\tau}}$.
\end{theorem}

\begin{proof}
See Section~\ref{sec-proof-conv}.
\end{proof}

According to Theorem~\ref{theorem:convergence_for_convex_functions}, for achieving accuracy $\varepsilon > 0$ in terms of the expected value of the objective, one needs the following number of iterations:
$$
K_\tau := \frac{2 D_\tau^2}{\varepsilon}.
$$
In particular, for $\tau = 1$, we have $D_{\tau}^2 = \Tr(B) D^2$, where $D$ is the radius of the sublevel set $L_f(x_0)$ measured in the standard Euclidean norm; this recovers the already known result for the RCD method \cite{nesterov2012efficiency}.

Let us fix $1 \leq \tau_1 < \tau_2 \leq \Rank(B)$ and compare the efficiency estimates of RCDVS with $\tau_1$ coordinates with that of $\tau_2$ coordinates. We obtain that
$$
\frac{K_{\tau_1}}{K_{\tau_2}} = \frac{D_{\tau_1}^2}{D_{\tau_2}^2}.
$$
Thus, we need to compare the quantities $D_{\tau_1}^2$ and $D_{\tau_2}^2$. By Lemma~\ref{lemma:relations_between_different_tau}, we have
$$
\| x - x^* \|_{B_{\tau_2}}^2 \leq \| x - x^* \|_{B_{\tau_1}}^2 \leq R_{\lambda}(\tau_1, \tau_2) \| x - x^* \|_{B_{\tau_2}}^2
$$
for all $x, x^* \in \R^n$. Hence, by first minimizing in $x^* \in \Argmin f$, and then maximizing in $x \in L_f(x_0)$, we obtain
$$
D_{\tau_2}^2 \leq D_{\tau_1}^2 \leq R_{\lambda}(\tau_1, \tau_2) D_{\tau_2}^2.
$$
This means that the method with a bigger number of coordinates is always not worse than the corresponding method with a smaller number of coordinates, but it can also be faster up to the ratio \eqref{eq:theoretical_acceleration_ratio}.

\subsection{Strongly convex functions}\label{section:convergence_results:strongly_convex_functions}

Now let us consider the strongly convex case. For measuring the parameter of strong convexity, it is natural to use the norm $\| \cdot \|_{B_\tau}$. Recall that a differentiable function $f : \R^n \to \R$ is called \emph{strongly convex} with respect to the norm $\| \cdot \|_{B_\tau}$ if there exists $\mu_{\tau} > 0$ such that
\begin{equation}\label{eq:strong_convexity}
f(y) \geq f(x) + \langle \nabla f(x), y - x \rangle + \frac{\mu_\tau}{2} \| y - x \|_{B_\tau}^2
\end{equation}
for all $x, y \in \R^n$. The largest possible value of $\mu_{\tau}$, satisfying \eqref{eq:strong_convexity}, is called the \emph{modulus} of strong convexity. Observe that if $f$ is additionally 1-smooth with respect to $\| \cdot \|_B$ (see \eqref{eq:smoothness_wrt_B}), then we must have $\mu_\tau B_\tau \preceq B$. Thus, $B$ cannot be degenerate in this situation.

\begin{theorem}[Convergence rate for strongly convex functions]\label{theorem:convergence_for_strongly_convex_functions}
Let $f : \R^n \to \R$ be a function which is 1-smooth with respect to the norm $\| \cdot \|_B$, where $B$ is an $n \times n$ real symmetric positive definite matrix, let $1 \leq \tau \leq n$, and let $f$ be strongly convex with respect to the norm $\| \cdot \|_{B_{\tau}}$ with modulus $\mu_\tau$. Let $x_0$ be a deterministic point in $\R^n$, let $K \geq 1$, and let $(x_k)_{k=1}^K$ be the random points in $\R^n$ generated by $\RCDVS(f, B, \tau, x_0, K)$. Then
\begin{equation}\label{eq:convergence_for_strongly_convex_functions}
\E f(x_k) - \min f \leq (1 - \mu_\tau)^k (f(x_0) - \min f)
\end{equation}
for all $0 \leq k \leq K$.
\end{theorem}

\begin{proof}
See Section~\ref{sec-proof-sconv}.
\end{proof}

Since $1 - \mu_\tau \leq e^{-\mu_\tau}$, this means that, given any $\varepsilon > 0$, for achieving accuracy $\varepsilon$ in terms of the expected value of the objective, one needs the following number of iterations:
$$
K_{\tau} := \frac{1}{\mu_\tau} \ln \frac{f(x_0) - \min f}{\varepsilon}.
$$
For $\tau=1$, we have $\mu_\tau = \Tr(B) \mu$, where $\mu$ is the strong convexity parameter of $f$ in the standard Euclidean norm. This recovers the convergence rate of the RCD method for strongly convex functions from \cite{nesterov2012efficiency}.

Similarly to the discussion in Section~\ref{section:convergence_results:convex_functions}, let us compare the efficiency estimates for different values of $\tau$. Fix $1 \leq \tau_1 < \tau_2 \leq n$. Then, the acceleration rate equals
$$
\frac{K_{\tau_1}}{K_{\tau_2}} = \frac{\mu_{\tau_2}}{\mu_{\tau_1}}.
$$
Let us compare the constants $\mu_{\tau_1}$ and $\mu_{\tau_2}$. By Lemma~\ref{lemma:relations_between_different_tau}, for all $h \in \R^n$, we have
\begin{equation}\label{str-conv-aux}
\| h \|_{B_{\tau_2}}^2 \leq \| h \|_{B_{\tau_1}}^2 \leq R_{\lambda}(\tau_1, \tau_2) \| h \|_{B_{\tau_2}}^2
\end{equation}
Hence, by strong convexity of $f$ in the $B_{\tau_1}$-norm, it follows that
$$
f(y) - f(x) - \langle \nabla f(x), y - x \rangle \geq \frac{\mu_{\tau_1}}{2} \| y - x \|_{B_{\tau_1}}^2 \geq \frac{\mu_{\tau_1}}{2} \| y - x \|_{B_{\tau_2}}^2
$$
for all $x, y \in \R^n$. This means that the modulus of strong convexity of $f$ with respect to $\| \cdot \|_{B_{\tau_2}}$ is at least $\mu_{\tau_1}$, i.e. $\mu_{\tau_2} \geq \mu_{\tau_1}$. Similarly, combining the strong convexity of $f$ in the $B_{\tau_2}$-norm with the second inequality in \eqref{str-conv-aux}, we obtain
$$
f(y) - f(x) - \langle \nabla f(x), y - x \rangle \geq \frac{\mu_{\tau_2}}{2} \| y - x \|_{B_{\tau_2}}^2 \geq \frac{\mu_{\tau_2}}{2 R_{\lambda}(\tau_1, \tau_2)} \| y - x \|_{B_{\tau_1}}^2
$$
for all $x \in \R^n$. This means that $\mu_{\tau_1} \geq \frac{\mu_{\tau_2}}{R_{\lambda}(\tau_1, \tau_2)}$. Hence, our reasoning shows that
$$
\mu_{\tau_1} \leq \mu_{\tau_2} \leq R_{\lambda}(\tau_1, \tau_2) \mu_{\tau_1}.
$$
Thus, we obtain absolutely the same result as in the previous section: the efficiency of the method monotonically improves with $\tau$ and the acceleration factor can reach the ratio \eqref{eq:theoretical_acceleration_ratio} (for example, one can verify that this is the case for a strictly convex quadratic function).

\begin{remark}
The results, presented in this section, can be straightforwardly extended from strongly convex functions to a more broader class of \emph{gradient dominated functions of degree 2}, also known as the functions satisfying the \emph{Polyak--\L{}ojasiewicz condition}. For more information and different examples of such functions, see \cite{karimi2016linear}. Note also that recently, in the context of coordinate descent methods, there has appeared an even more general condition, called \emph{Generalized Error Bound Property} \cite{necoara2016parallel}. However, we do not know whether our results can be extended for this property.
\end{remark}

\subsection{Proof of Lemma~\ref{lemma:sum_of_adjugates_of_principal_submatrices}}\label{section:proof_of_key_estimate}

In this section, we give the proof of Lemma~\ref{lemma:sum_of_adjugates_of_principal_submatrices} assuming that $n \geq 2$ (otherwise the claim is trivial). We start with introducing a little new notation that will be used only inside this section. For a subset $S \in {[n] \choose \tau}$, by $I_{-S}$ we denote the $n \times (n-\tau)$ matrix obtained from the $n \times n$ identity matrix $I$ by removing the columns with indices from $S$ (i.e. $I_{-S} := I_{[n] \setminus S}$). For an $n \times n$ matrix $B$ and a subset $S \in {[n] \choose \tau}$, by $B_{-S \times -S}$ we denote the $(n-\tau) \times (n-\tau)$ submatrix obtained from $B$ by removing the rows and columns with indices from $S$ (i.e. $B_{-S \times -S} := I_{-S}^T B I_{-S}$); similarly, for a vector $v \in \R^n$, by $v_{-S}$ we denote the subvector of size $n-\tau$ obtained from $v$ by removing the elements with indices from $S$ (i.e. $v_{-S} := I_{-S}^T v$); for brevity, for each $1 \leq i \leq n$, we also use $I_{-i}$, $B_{-i \times -i}$ and $v_{-i}$ instead of more cumbersome $I_{-\{i\}}$, $B_{-\{i\} \times -\{i\}}$ and $v_{-\{i\}}$ respectively.

To prove Lemma~\ref{lemma:sum_of_adjugates_of_principal_submatrices}, let us consider the matrix-valued polynomial
\begin{equation}\label{eq:matrix_valued_polynomial}
t \in \R \mapsto P(t) := \Adj(B - t I)
\end{equation}
and show that the left- and right-hand sides of \eqref{eq:sum_of_adjugates_of_principal_submatrices} are, up to a constant multiplicative factor, different representations of the $(n-\tau)$-th derivative of $P$ at zero.

We start with the easier right-hand side. Using the spectral decomposition $B = Q \Diag(\lambda) Q^T$ and the definition of the adjugate matrix, for each $t \in \R$ we readily obtain the following spectral decomposition of $P$:
$$
P(t) = Q \Diag(d_1(t), \dots, d_n(t)) Q^T,
$$
where $d_i : \R \to \R$ for each $1 \leq i \leq n$ is the polynomial
$$
d_i(t) := \prod_{1 \leq j \leq n : j \neq i} (\lambda_j(B) - t).
$$
Opening the parentheses and grouping the terms by the powers of $t$, we see that
$$
d_i^{(n-\tau)}(0) = (-1)^{n-\tau} (n-\tau)! \sigma_{\tau-1}(\lambda_{-i})
$$
for each $1 \leq i \leq n$, and hence
\begin{equation}\label{eq:first_expression_for_derivative}
P^{(n-\tau)}(0) = (-1)^{n-\tau} (n-\tau)! Q \Diag(\sigma_{\tau-1}(\lambda_{-1}), \dots, \sigma_{\tau-1}(\lambda_{-n})) Q^T.
\end{equation}

Now we take another approach to calculating the derivative $P^{(n-\tau)}(0)$ by directly differentiating the original expression \eqref{eq:matrix_valued_polynomial}. The key inductive step here is
\begin{lemma}[Inductive step]\label{lemma:inductive_step}
Let $B$ be an $n \times n$ ($n \geq 2$) real symmetric matrix, let $P$ be the matrix-valued polynomial \eqref{eq:matrix_valued_polynomial}, and let $t \in \R$. Then
\begin{equation}\label{eq:inductive_step}
P'(t) = -\sum_{i=1}^n I_{-i} \Adj(B_{-i \times -i} - t I) I_{-i}^T.
\end{equation}
\end{lemma}

Suppose for the moment that Lemma~\ref{lemma:inductive_step} holds, and let $t \in \R$ be arbitrary. Differentiating both sides of \eqref{eq:inductive_step} (each time applying Lemma~\ref{lemma:inductive_step} to the matrix $B_{-i \times -i}$) and assuming that $n \geq 3$, we obtain
$$
\begin{aligned}
P''(t) &= \sum_{i=1}^n \sum_{1 \leq j \leq n : j \neq i} I_{-\{i, j\}} \Adj(B_{-\{i, j\} \times -\{i, j\}} - t I) I_{-\{i, j\}}^T \\
&= 2 \sum_{1 \leq i < j \leq n} I_{-\{i, j\}} \Adj(B_{-\{i, j\} \times -\{i, j\}} - t I) I_{-\{i, j\}}^T
\end{aligned}
$$
Similarly, assuming that $n \geq 4$, we have
$$
\begin{aligned}
P'''(t) &= -2 \sum_{1 \leq i < j \leq n} \sum_{1 \leq k \leq n : k \neq i, j} I_{-\{i, j, k\}} \Adj(B_{-\{i, j, k\} \times -\{i, j, k\}} - t I) I_{-\{i, j, k\}}^T \\
&= -6 \sum_{1 \leq i < j < k \leq n} I_{-\{i, j, k\}} \Adj(B_{-\{i, j, k\} \times -\{i, j, k\}} - t I) I_{-\{i, j, k\}}^T,
\end{aligned}
$$
and, more generally (by induction on $r$), that
$$
P^{(r)}(t) = (-1)^r r! \sum_{S \in {[n] \choose r}} I_{-S} \Adj(B_{-S \times -S} - t I) I_{-S}^T
$$
for all $0 \leq r \leq n - 1$. In particular,
\begin{equation}\label{eq:second_expression_for_derivative}
P^{(n-\tau)}(0) = (-1)^{n-\tau} (n-\tau)! \sum_{S \in {[n] \choose \tau}} I_S \Adj(B_{S \times S}) I_S^T.
\end{equation}
Equating \eqref{eq:first_expression_for_derivative} and \eqref{eq:second_expression_for_derivative}, we obtain the claim of Lemma~\ref{lemma:sum_of_adjugates_of_principal_submatrices}.

All that remains is to prove Lemma~\ref{lemma:inductive_step}.
\begin{proof}
We begin with a couple of technical simplifications. First, it suffices to prove the claim only for $t = 0$; the general case then follows by replacing $B$ with $B - t I$. Second, we can assume that $B$ and all its $(n-1) \times (n-1)$ principal submatrices ($B_{-1, -1}, \dots, B_{-n, -n}$) are simultaneously non-degenerate; otherwise we can replace $B$ with $B + \delta I$ for an arbitrary sufficiently small $\delta > 0$ (such that $B + \delta I$ satisfies the above requirement) and then pass to the limit as $\delta \to 0$ using the continuity of the both sides of \eqref{eq:inductive_step} in $\delta$.

Since $B$ is non-degenerate, there exists a sufficiently small neighborhood around zero such that $B - t I$ is non-degenerate for all $t$ from this neighborhood and
$$
P(t) = \Det(B - t I) (B - t I)^{-1}
$$
by Cramer's rule. Differentiating and denoting $C := B^{-1}$, we obtain
$$
P'(0) = -\Det(B) (\Tr(B^{-1}) B^{-1} - B^{-2}) = -\Det(B) \sum_{i=1}^n ( C_{ii} C - c_i c_i^T ),
$$
where $c_1, \dots, c_n \in \R^n$ are the columns of $C$. Observe that the $i$-th row and the $i$-th column of the matrix $C_{ii} C - c_i c_i^T$ for each $1 \leq i \leq n$ consist entirely of zeros, so
$$
P'(0) = -\Det(B) \sum_{i=1}^n I_{-i} ( C_{ii} C_{-i \times -i} - (c_i)_{-i} (c_i)_{-i}^T ) I_{-i}^T.
$$

Thus, to obtain the claim, it suffices to demonstrate that
$$
\Det(B) ( C_{ii} C_{-i \times -i} - (c_i)_{-i} (c_i)_{-i}^T ) = \Adj(B_{-i \times -i})
$$
for all $1 \leq i \leq n$. Replacing $B$ with $P^T B P$ if necessary (where $P$ is the $n \times n$ permutation matrix obtained from the identity matrix by moving the $i$-th column into the end), it is enough to consider only the case $i = n$. Let
$$
B = \left( \begin{array}{cc} F & z \\ z^T & \alpha \end{array} \right),
$$
where $F$ is the top-left $(n-1) \times (n-1)$ principal submatrix of $B$, $z \in \R^{n-1}$ is the right-most column of $B$ with the last element removed, and $\alpha := B_{nn}$ is the element in the lower right corner. Note that $F$ is non-degenerate as a principal $(n-1) \times (n-1)$ submatrix of $B$ (by the technical assumption made at the very beginning). Using the formula for inverting a block matrix, we obtain that $\alpha - \langle F^{-1} z, z \rangle \neq 0$, and
$$
C = B^{-1} = \left( \begin{array}{cc} F^{-1} + \frac{ F^{-1} z z^T F^{-1} }{ \alpha - \langle F^{-1} z, z \rangle } & -\frac{F^{-1} z} { \alpha - \langle F^{-1} z, z \rangle } \\ -\frac{ z^T F^{-1} }{ \alpha - \langle F^{-1} z, z \rangle } & \frac{1}{ \alpha - \langle F^{-1} z, z \rangle } \end{array} \right).
$$
In particular, we see that
$$
C_{nn} C_{-n \times -n} - (c_n)_{-n} (c_n)_{-n}^T = \frac{ F^{-1} }{ \alpha - \langle F^{-1} z, z \rangle }.
$$
Since by Cramer's rule
$$
\Adj(B_{-n \times -n}) = \Adj(F) = \Det(F) F^{-1},
$$
it remains to check whether
$$
\Det(B) = \Det(F) ( \alpha - \langle F^{-1} z, z \rangle ).
$$
But this is exactly the formula for the determinant of a block matrix.
\end{proof}

\section{Implementation of volume sampling}\label{section:implementation_of_volume_sampling}

Let $B$ be an $n \times n$ real symmetric positive semidefinite matrix, and let $1 \leq \tau \leq \Rank(B)$. In this section, we discuss how to generate a random variable $S_0 \sim \Vol_{\tau}(B)$ according to volume sampling.

\subsection{General algorithm}\label{section:general_algorithm}

Recall that volume sampling is sampling with a finite number of outcomes. Thus, in principle, $S_0$ can be generated by any general method for generating random variables taking a finite number of values. Let us briefly review one such method which is based on the following result:
\begin{proposition}[Generating a random variable taking a finite number of values]\label{proposition:generating_discrete_random_variable}
Let $X := \{x_1, \dots, x_N\}$ be a finite set, and let $p_1, \dots, p_N$ be non-negative numbers such that $\sum_{k=1}^N p_k = 1$. For each $1 \leq k \leq N$, let $P_k := \sum_{k'=1}^k p_{k'}$. Let $u$ be a random variable uniformly distributed on the interval $(0, 1)$, and let $\xi := x_{k_0}$, where $k_0 := \min\{ 1 \leq k \leq N : u \leq P_k \}$. Then $\xi$ is a well-defined random variable taking values in $X$ such that $\Pr(\xi = x_k) = p_k$ for all $1 \leq k \leq N$.
\end{proposition}

\begin{proof}
For convenience, denote $P_0 := 0$. Observe that $0 = P_0 \leq P_1 \leq \dots \leq P_N = 1$. Thus, $k_0 = k$ for some $1 \leq k \leq N$ iff $u$ belongs to the interval $(P_{k-1}, P_k]$. Since these intervals are disjoint and their union is $(0, 1]$, the variable $k_0$ is well-defined. Hence, $\xi$ is well-defined, and $\Pr(\xi = x_k) = \Pr(k_0 = k) = \Pr(P_{k-1} < u \leq P_k) = P_k - P_{k-1} = p_k$ for each $1 \leq k \leq N$.
\end{proof}

Proposition~\ref{proposition:generating_discrete_random_variable} is in fact a two-stage algorithm for generating a random variable $\xi$ taking values in $\{x_1, \dots, x_N\}$ given the corresponding list of probabilities $p_1, \dots, p_N$. At the first stage of this algorithm (called \emph{preprocessing}), we compute the cumulative sums $P_1, \dots, P_N$. This requires $O(N)$ operations. At the second stage (called \emph{sampling}), we first generate a random variable $u$ uniformly distributed on $(0, 1)$, then compute $k_0$, and finally output $\xi := x_{k_0}$. Since the cumulative sums $(P_k)_{1 \leq k \leq N}$ are monotonically increasing, one can use binary search for efficiently finding $k_0$. Thus, the complexity of sampling is just $O(\ln N)$ operations. Note that the preprocessing has to be done only once; after this, one can generate arbitrarily many independent samples using the sampling routine.

In the case of volume sampling, the above procedure looks as follows. At the preprocessing stage, we iterate over all $N = {n \choose \tau}$ possible $\tau$-element subsets of $[n]$, computing the corresponding principal minors of $B$ and corresponding cumulative sums. This requires $O({n \choose \tau} \tau^3)$ operations in total assuming that the complexity of calculating a minor of size $\tau$ is $O(\tau^3)$. During sampling, we use binary search to find the number $k_0$, and then return the set $S_0 \in {[n] \choose \tau}$ corresponding to $k_0$. Each sampling thus requires $O(\ln{n \choose \tau})$ operations.

Unfortunately, the above $O({n \choose \tau} \tau^3)$ preprocessing time makes the general algorithm impractical for most values of $\tau$. Nevertheless, it is still applicable for several very small values of $\tau$. For example, when $\tau=1$, the preprocessing time and memory complexities are both $O(n)$, while the sampling time and memory complexities are $O(\ln n)$ and $O(1)$ respectively. Another interesting regime is $\tau=2$. In this case, the preprocessing time and memory complexities are both $O(n^2)$, while the sampling time and memory complexities are the same as before. Note that in many applications the objective function $f : \R^n \to \R$ has the form $f(x) := \phi(A x, x)$, where $A$ is a real $m \times n$ matrix, and $\phi : \R^m \times \R^n \to \R$ is a function that can be computed in time $O(m + n)$ (see Section~\ref{section:examples_of_applications} for different examples). In these applications, the $O(n^2)$ memory is comparable to the cost of storing $A$, and the $O(n^2)$ time is comparable to the cost of one computation of the objective and is often allowable (see also Section~\ref{section:sparse_two_element_volume_sampling} for a possible treatment of sparsity).

While the general procedure described above is not polynomial, we note that there are more specialized methods for volume sampling that are polynomial (e.g. see \cite{deshpande2010efficient} for an exact algorithm and several efficient approximate ones). However, they are designed for generating just one sample and therefore are not directly suited for using them inside optimization algorithms, where one needs a very fast sampling routine that can be called at each iteration. Perhaps, it is possible to modify these methods by properly splitting them into a polynomial preprocessing stage and an independent sampling stage which is much faster, but we have not investigated this direction.

\subsection{Two-element volume sampling for sparse matrices}\label{section:sparse_two_element_volume_sampling}

Now suppose that the matrix $B$ is sparse and our goal is to implement 2-element volume sampling. The general algorithm described above requires $O(n^2)$ time and $O(n^2)$ memory for preprocessing, which may be too expensive when $n$ is large. In this section, we present a special method that takes into account the sparsity of $B$ and whose preprocessing time and memory complexities are both $O(\nnz(B) + n)$, where $\nnz(B)$ is the number of non-zero elements of $B$. When $B$ is dense, $\nnz(B) = n^2$, but it can be much smaller than $n^2$ if $B$ is sufficiently sparse. Once the preprocessing is done, each sampling then has the $O(\ln n)$ time complexity and the $O(1)$ memory complexity, which are exactly the same complexities as those of the general algorithm from the previous section (see Figure~\ref{fig:sampling_complexity} for a comparison).

\begin{figure}\centering
\begin{tabular}{l|l|l|l|l}
& \multicolumn{2}{c|}{Preprocessing} & \multicolumn{2}{c}{Sampling} \\ \hline
& Time & Memory & Time & Memory \\ \hline
Dense & $O(n^2)$ & $O(n^2)$ & $O(\ln n)$ & $O(1)$ \\ \hline
Sparse & $O(\nnz(B) + n)$ & $O(\nnz(B) + n)$ & $O(\ln n)$ & $O(1)$
\end{tabular}
\caption{\label{fig:sampling_complexity}Time and memory complexities of 2-element volume sampling for a dense and sparse matrix.}
\end{figure}

Assume that the matrix $B$ is given in the \emph{CSR (Compressed Sparse Row) format}\footnote{Strictly speaking, the classical CSR format is different from the one we are describing here. In the classical CSR format, the index vectors $j^{(1)}, \dots, j^{(n)}$ are concatenated into one large vector, similarly the value vectors $v^{(1)}, \dots, v^{(n)}$ are concatenated into one large vector, and instead the numbers $r_1, \dots, r_n$ one stores their cumulative sums. Nevertheless, the format we are describing is algorithmically equivalent to the original CSR format and one can be transformed into the other in a straightforward manner without any time or memory overhead.}: for each $1 \leq i \leq n$, we know the $r_i$ indices (possibly zero) $i \leq j^{(i)}_1 < \dots < j^{(i)}_{r_i} \leq n$ of all non-zero elements in the $i$-th row of $B$ which are located to the right of the diagonal (thus, $\{ j^{(i)}_1, \dots, j^{(i)}_{r_i} \} = \{ i \leq j \leq n : B_{ij} \neq 0 \}$), as well as the corresponding values $v^{(i)}_1, \dots, v^{(i)}_{r_i}$ of these elements (thus, $v^{(i)}_k := B_{i j^{(i)}_k}$ for all $1 \leq k \leq r_i$). For notational convenience, for each $1 \leq i < j \leq n; 1 \leq k \leq r_i$ we also define $j^{(i)}_{r_i+1} := n + 1$ and
$$
P_{v, h, t}(i, j, k) := v^{(i)}_1 (t_i - t_{j+1}) - h^{(i)}_k,
$$
where $h^{(1)} \in \R^{r_1}, \dots, h^{(n-1)} \in \R^{r_{n-1}}$, $t \in \R^{n+1}$ are given, and $h := (h^{(1)}, \dots, h^{(n-1)})$.

\begin{algorithm}
\caption{\label{algorithm:sparsetwovspreprocess}$\sparsetwovspreprocess(B)$}
\begin{algorithmic}[1]
\REQUIRE $B$: $n \times n$ real symmetric positive semidefinite matrix of rank at least two specified by its CSR-format $(r_1, \dots, r_n, j^{(1)}, \dots, j^{(n)}, v^{(1)}, \dots, v^{(n)})$.
\STATE For each $1 \leq i \leq n-1$, compute $h^{(i)} \in \R^{r_i}$, defined by $h^{(i)}_k := \sum_{k'=1}^k (v^{(i)}_{k'})^2$.
\STATE Compute $t \in \R^{n+1}$, defined by $t_i := \sum_{j=i}^n v^{(j)}_1$ (and $t_{n+1} := 0$).
\STATE Compute $q \in \R^{n-1}$, defined by $q_i := \sum_{i'=1}^i P_{v, h, t}(i', n, r_{i'})$.
\RETURN $(h^{(1)}, \dots, h^{(n-1)}, t, q)$
\end{algorithmic}
\end{algorithm}

\begin{algorithm}
\caption{\label{algorithm:sparsetwovssample}$\sparsetwovssample(B, h^{(1)}, \dots, h^{(n-1)}, t, q)$}
\begin{algorithmic}[1]
\REQUIRE $B$: $n \times n$ real symmetric positive semidefinite matrix of rank at least two specified by its CSR-format $(r_1, \dots, r_n, j^{(1)}, \dots, j^{(n)}, v^{(1)}, \dots, v^{(n)})$; $h^{(1)} \in \R^{r_1}, \dots, h^{(n-1)} \in \R^{r_{n-1}}$; $t \in \R^{n+1}$; $q \in \R^{n-1}$.
\STATE Independently generate random $u_1, u_2$ uniformly distributed on $(0, 1)$.
\STATE Find $i_0 := \min\{ 1 \leq i \leq n - 1 : u_1 \leq \frac{q_i}{q_{n-1}} \}$ using binary search.
\STATE Find $k_l := \min\{ 1 \leq k \leq r_{i_0} : u_2 \leq \frac{P_{v, h, t}(i_0, j^{(i_0)}_{k+1} - 1, k)}{P_{v, h, t}(i_0, n, r_{i_0})} \}$ using binary search.
\STATE Find $j_0 := \min\{ j^{(i_0)}_{k_l} \leq j \leq j^{(i_0)}_{k_l + 1} - 1 : u_2 \leq \frac{P_{v, h, t}(i_0, j, k_l)}{P_{v, h, t}(i_0, n, r_{i_0})} \}$ using binary search.
\RETURN $\{ i_0, j_0 \}$
\end{algorithmic}
\end{algorithm}

Now the notation has been established and we are ready to present the algorithm. Similarly to the method from the previous section, it is a two-stage procedure that consists of an expensive preprocessing stage (Algorithm~\ref{algorithm:sparsetwovspreprocess}) and a cheap sampling stage (Algorithm~\ref{algorithm:sparsetwovssample}) that can be executed as many times as one wishes once the preprocessing has terminated.

Clearly, the time and memory complexities of the preprocessing stage are both $O(\nnz(B) + n)$. The sampling stage consists of three successive binary searches over certain subsets of $[n]$ and thus has the $O(\ln n)$ time complexity and the $O(1)$ memory complexity when properly implemented (the function $P_{v, h, t}$ should be computed on the fly inside each binary search).

We now prove the correctness of the algorithm.

\begin{theorem}
Let $B$ be an $n \times n$ real symmetric positive semidefinite matrix with rank at least two. Let $(h^{(1)}, \dots, h^{(n-1)}, t, q) := \sparsetwovspreprocess(B)$, and let $S_0 := \sparsetwovssample(B, h^{(1)}, \dots, h^{(n-1)}, t, q)$. Then $S_0$ is a well-defined random variable distributed according to $\Vol_2(B)$.
\end{theorem}

\begin{proof}
Observe that, for each $1 \leq i \leq n - 1$ and each $1 \leq k \leq r_i$, we have
$$
h^{(i)}_k = \sum_{j=i}^{j^{(i)}_k} B_{ij}^2, \qquad h^{(i)}_{r_i} = \sum_{j=i}^n B_{ij}^2, \qquad v^{(i)}_1 = B_{ii}, \qquad t_i = \sum_{j=i}^n B_{jj}.
$$
For each $1 \leq i < j \leq n$, denote
$$
P(i, j) := \sum_{j'=i+1}^j \Det(B_{ \{i,j\} \times \{i,j\} })
$$
Then, for each $1 \leq i \leq n - 1$, it holds that
$$
v^{(i)}_1 t_i - h^{(i)}_{r_i} = \sum_{j=i}^n (B_{ii} B_{jj} - B_{ij}^2) = P(i, n).
$$
From Proposition~\ref{proposition:generating_discrete_random_variable}, it follows that $i_0$ is a well-defined random variable taking values in $\{1,\dots,n-1\}$ with probabilities
\begin{equation}\label{eq:sparse_22_volume_sampling_well_defined:proof:2}
\Pr(i_0 = i) = \frac{P(i, n)}{\sum_{i'=1}^{n-1} P(i', n)}.
\end{equation}

Next observe that, for each $1 \leq i \leq n - 1$ and $1 \leq k \leq r_i$, we have
$$
P(i, j, k) = B_{i i} \sum_{j'=i}^j B_{j' j'} - \sum_{j'=i}^{j_k^{(i)}} B_{i j'}^2 = \sum_{j'=i}^j (B_{ii} B_{j'j'} - B_{i j'}^2) = P(i, j)
$$
for all $j^{(i)}_k \leq j \leq j^{(i)}_{k+1} - 1$. In particular, $P(i_0, j^{(i_0)}_{k+1} - 1, k) = P(i_0, j^{(i_0)}_{k+1} - 1)$ for all $1 \leq k \leq r_{i_0}$, from which we see that $P(i_0, j^{(i_0)}_{k+1} - 1, k)$ monotonically increases with $k$ (since $j^{(i_0)}_{k+1}$ does so and $P(i_0, j)$ monotonically increases with $j$) until it reaches $P(i_0, n, r_{i_0}) = P(i_0, n)$ when $k = r_{i_0}$. This shows that $k_l$ is well-defined. Similarly, we can write $P(i_0, j, k_l) = P(i_0, j)$ for all $j^{(i_0)}_{k_l} \leq j \leq j^{(i_0)}_{k_l+1} - 1$, from which it follows that $P(i_0, j, k_l)$ monotonically increases with $j$. Combining this with the definition of $k_l$, which gives
$$
\frac{P(i_0, j^{(i_0)}_{k_l})}{P(i, n)} \leq u_2 \leq \frac{P(i_0, j^{(i_0)}_{k_l+1} - 1)}{P(i, n)},
$$
we conclude that $j_0$ is well-defined and in fact
$$
j_0 = \min\left\{ i_0 + 1 \leq j \leq n : u_2 \leq \frac{P(i_0, j)}{P(i_0, n)} \right\}.
$$

Applying Proposition~\ref{proposition:generating_discrete_random_variable} again, we obtain that, conditioned on $i_0 = i$, the random variable $j_0$ takes values in $\{ i + 1, \dots, n \}$ such that
$$
\Pr(j_0 = j \mid i_0 = i) = \frac{\Det(B_{\{i,j\} \times \{i,j\}})}{P(i, n)}.
$$
for all $1 \leq i < j \leq n$. Undoing the conditioning, we see that $(i_0, j_0)$ is a well-defined random variable taking values in $\{ (i, j) : 1 \leq i < j \leq n \}$ with probabilities
$$
\Pr(i_0 = i; j_0 = j) = \Pr(i_0 = i) \Pr(j_0 = j \mid i_0 = i) = \frac{\Det(B_{\{i, j\} \times \{i, j\}})}{\sum_{i'=1}^{n-1} P(i', n)},
$$
and the claim follows.
\end{proof}

\section{Examples of applications}\label{section:examples_of_applications}

Now we consider several examples of objective functions for which it is possible to apply the RCDVS method and discuss different implementation details.

\subsection{Quadratic function}\label{section:quadratic_function}

Our first example of an objective function is the convex quadratic $f : \R^n \to \R$, defined by
$$
f(x) := \frac{1}{2} \langle A x, x \rangle - \langle b, x \rangle,
$$
where $A$ is a given $n \times n$ real symmetric positive semidefinite matrix and $b \in \R^n$. This function is 1-smooth with respect to the seminorm $\| \cdot \|_A$, so one can minimize it by RCDVS with
$$
B := A.
$$
For doing volume sampling, one can either use the general algorithm from Section~\ref{section:general_algorithm} when the matrix $A$ is dense, or the special one from Section~\ref{section:sparse_two_element_volume_sampling} when $A$ is sparse.

\subsection{Separable problems}\label{section:separable_problems}

The second example gives rise to a whole family of objective functions $f : \R^n \to \R$ that are admissible for the RCDVS method and can be obtained by composing some smooth separable function with a linear mapping:
$$
f(x) := \sum_{i=1}^m g_i(\langle a_i, x \rangle).
$$
Here $a_1, \dots, a_m \in \R^n$ are given vectors and $g_1, \dots, g_m : \R \to \R$ are univariate functions such that $g_i$ is $L_i$-smooth $(L_i \geq 0)$ for each $1 \leq i \leq n$, meaning that it is differentiable and satisfies
$$
g_i(t) \leq g_i(t_0) + g_i'(t_0) (t - t_0) + \frac{L_i}{2} (t - t_0)^2
$$
for all $t, t_0 \in \R$. It is easy to see from the definitions that the resulting function $f$ turns out to be 1-smooth with respect to the seminorm $\| \cdot \|_B$, where
\begin{equation}\label{eq:matrix_for_separable_problems}
B := \sum_{i=1}^m L_i a_i a_i^T.
\end{equation}

\begin{example}[Least squares]
Let $f$ be the least squares function
$$
f(x) := \frac{1}{2} \sum_{i=1}^m (\langle a_i, x \rangle - b_i)^2,
$$
where $b_1, \dots, b_m \in \R$. In this case, $g_i(t) := \frac{1}{2} (t - b_i)^2$ is the quadratic function with $L_i = 1$ for each $1 \leq i \leq m$.
\end{example}

\begin{example}[Logistic regression]
Let $f$ be the logistic regression function
$$
\sum_{i=1}^m \ln(1 + e^{-b_i \langle a_i, x \rangle}),
$$
where $b_1, \dots, b_m \in \{ -1, 1 \}$. In this case, $g_i(t) := \ln(1 + e^{-b_i t})$ is the logistic function with $L_i = \frac{1}{4}$ for each $1 \leq i \leq m$.
\end{example}

The matrix $B$ can be computed in $O(m n^2)$ operations. If $n$ is sufficiently small, this can be done rather efficiently, and then one can apply RCDVS for several small values of $\tau$ (e.g. $\tau=2, 3$, etc.).

If $n$ is large, one can still use RCDVS provided that the vectors $a_1, \dots, a_m$ are \emph{sparse}. Indeed, observe that the number of non-zero elements in $B$ is bounded above by $\sum_{i=1}^m p_i$, where $p_i$ for each $1 \leq i \leq n$ denotes the number of non-zero elements in $a_i$. Furthermore, a sparse representation of $B$ (e.g. the commonly used sparse compressed row/column formats) can be obtained in
\begin{equation}\label{eq:compl_of_comp_B}
O\left( \sum_{i=1}^m p_i^2 + n \right)
\end{equation}
operations (possibly with some logarithmic terms when a further sorting of indices is needed). After this, one can use the efficient algorithm for sparse two-element volume sampling from Section~\ref{section:sparse_two_element_volume_sampling}, whose preprocessing complexity is the same as \eqref{eq:compl_of_comp_B}. For example, if for all $1 \leq i \leq n$, we have $p_i \leq p$, where $p$ is some sufficiently small integer, then both the computation of $B$ and the preprocessing procedure take
$$
O(m p^2 + n)
$$
operations, which is only linear in both $m$ and $n$.

\subsection{Smoothing technique}\label{section:smoothing_technique}

Another interesting and quite rich source of examples comes from the smoothing technique \cite{nesterov2005smooth,beck2012smoothing}, which we now briefly review. Let $g : \R^m \to \R$ be a convex function. By the Fenchel--Moreau theorem, we can write
$$
g(y) = \max_{s \in G_*} \{ \langle y, s \rangle - g^*(s) \}
$$
for all $y \in \R^m$, where $g^* : G_* \to \R$ is the Fenchel conjugate of $g$ with the effective domain $G_*$ (assume that $G_*$ is bounded). Let $\omega^* : \Omega_* \to \R$ be a distance generating function on $G_*$ with respect to the standard Euclidean norm $\| \cdot \|$ (i.e. a non-negative closed convex function with domain $\Omega_* \supseteq G_*$ that is 1-strongly convex on $G_*$ with respect to $\| \cdot \|$). Let $\mu > 0$, and let $g_{\mu} : \R^n \to \R$ be the function
$$
g_\mu(y) := \max_{s \in G_*} \{ \langle y, s \rangle - g^*(s) - \mu \omega^*(s) \}.
$$
It is known that $g_\mu$ satisfies $g_\mu(y) \leq g(y) \leq g_{\mu}(y) + \mu \max_{G_*} \omega^*$ for all $y \in \R^m$, and moreover it is $\frac{1}{\mu}$-smooth with respect to $\| \cdot \|$. Thus, $g_\mu$ can be seen as a smooth uniform approximation of $g$, where the parameter $\mu$ controls both the accuracy of approximation and its level of smoothness.

Now let $A$ be an $m \times n$ real matrix, let $b \in \R^n$, and define $f, f_{\mu} : \R^n \to \R$ by
$$
f(x) := g(A x - b), \qquad f_\mu(x) := g_{\mu}(A x - b).
$$
It is easy to see that $f_\mu(x) \leq f(x) \leq f_\mu(x) + \mu \max_{G_*} \omega^*$ for all $x \in \R^n$. Furthermore, the function $f_\mu$ is $\frac{1}{\mu}$-smooth with respect to the seminorm $\| \cdot \|_{A^T A}$. Thus, the problem of minimizing $f$ can be replaced with the problem of minimizing its smooth approximation $f_\mu$ for some carefully chosen value of $\mu$. This latter problem can be solved by RCDVS with
$$
B := \frac{1}{\mu} A^T A.
$$
This matrix has exactly the same structure as the one from \eqref{eq:matrix_for_separable_problems}.

\begin{example}
The function
$$
f_{\mu}(x) := \mu \ln\left( \sum_{i=1}^m e^{(A x - b)_i / \mu} \right) - \mu \ln m
$$
is obtained from $g(y) := \max\{ y_1, \dots, y_m \}$ using the negative entropy function $\omega^*(s) := \sum_{i=1}^m s_i \ln s_i$ with domain $\Omega_* := \{ s \in \R^m : \sum_{i=1}^m s_i = 1; \, s_1, \dots, s_m \geq 0 \}$.
\end{example}

\begin{example}\label{example:huber_function}
The function
$$
f_{\mu}(x) := \sum_{i=1}^m H_{\mu}((A x - b)_i),
$$
where $H_{\mu} : \R \to \R$ is the \emph{Huber function}
$$
H_{\mu}(t) := \begin{cases}
\frac{t^2}{2 \mu}, & \text{if $|t| \leq \mu$}, \\
|t| - \frac{\mu}{2}, & \text{otherwise},
\end{cases}
$$
is obtained from the $l^1$-norm $g(y) := \| y \|_1$ using the Euclidean distance generating function $\omega^*(s) := \frac{1}{2} \| s \|^2$ with domain $\Omega_* := \R^m$.
\end{example}

\begin{example}
The function
$$
f_{\mu}(x) := \sqrt{\| A x - b \|^2 + \mu^2} - \mu
$$
is obtained from $g(y) := \| y \|$ using the function $\omega^*(s) := 1 - \sqrt{1 - \| s \|^2}$ with domain $\Omega_* := \{ s \in \R^n : \| s \| \leq 1 \}$.
\end{example}

\subsection{Combinations of previous examples}\label{section:combinations}

Finally, one can take non-negative linear combinations of the already considered examples. Indeed, let $f_1, \dots, f_r : \R^n \to \R$ be functions, where $f_i$ for each $1 \leq i \leq r$ is 1-smooth with respect to an $n \times n$ real symmetric positive semidefinite matrix $B_i$, and let $\alpha_1, \dots, \alpha_r > 0$. Then the sum $f := \sum_{i=1}^r \alpha_i f_i$ is 1-smooth with respect to the seminorm $\| \cdot \|_B$ with $B := \sum_{i=1}^r \alpha_i B_i$.

\begin{example}\label{example:l2_regularized_logistic_regression}
Let $f$ be the \emph{$l^2$-regularized logistic regression function}
$$
f(x) := \sum_{i=1}^m \ln(1 + e^{-b_i \langle a_i, x \rangle}) + \frac{\gamma}{2} \| x \|^2,
$$
where $a_1, \dots, a_m \in \R^n$, $b_1, \dots, b_m \in \{ -1, 1 \}$, $\gamma > 0$. In this case,
$$
B := \frac{1}{4} \sum_{i=1}^m a_i a_i^T + \gamma I.
$$
\end{example}

\section{Numerical experiments}\label{section:numerical_experiments}

In this section, we investigate the practical behavior of RCDVS and compare it with that of a couple of other already known methods.

The first one is the RCD method. Recall that it is in fact the same method as RCDVS with $\tau=1$. In comparing RCDVS with RCD, we are interested in investigating how the \emph{actual} acceleration ratio of RCDVS corresponds to our theoretical prediction (see \eqref{eq:theoretical_acceleration_ratio})
\begin{equation}\label{eq:acceleration_ratio_over_RCD}
R_{\lambda}(1, \tau) = \frac{\sum_{i=1}^n \lambda_i}{\sum_{i=\tau}^n \lambda_i}.
\end{equation}
The difference between the actual acceleration ratio and the theoretical one is that the latter is the ratio of the theoretical upper bounds on the performance of the methods, while the former is the ratio of the real number of iterations performed by the methods on a particular problem instance.

The second one, denoted SDNA, is Method 1 from \cite{qu2016sdna} which uses so-called $\tau$-nice sampling. As was already discussed in Section~\ref{section:related_work}, this method is exactly the same method as RCDVS with the only difference that it uses \emph{uniform} sampling (without replacement) instead of volume sampling. In comparing RCDVS with SDNA, we are interested in seeing how important is the sampling strategy to the performance of the general coordinate descent scheme that we consider in this paper.

\subsection{Quadratic function}

For the first experiment, we have chosen the convex quadratic function from Section~\ref{section:quadratic_function} and set $\tau=2$. Our goal is to observe how the behavior of the methods changes when the spectral gap between the two largest eigenvalues of $A$ increases. For this, we construct the matrix $A$ as follows. First, we choose some $\lambda_1 \geq \lambda_2 := 100$ and set $A := \Diag(\lambda_1, \lambda_2, 1, \dots, 1)$. Then we successively perform 10 random Householder reflections $A \mapsto (I - 2 u u^T) A (I - 2 u u^T)$ on the rows and columns of $A$, where each time the direction $u$ is sampled uniformly from the unit sphere in $\R^n$. Observe that, by construction, the eigenvalues of $A$ are exactly $\lambda_1, \lambda_2, 1, \dots, 1$. Once $A$ is constructed, we set $b := A x^*$, where $x^*$ is generated from the uniform distribution on the hypercube $[-1, 1]^n$ and run each method from $x_0 := 0$ until the objective value becomes $\varepsilon$-close to the optimal one for $\varepsilon := 0.01$. This procedure is repeated 10 times to take into account the randomness in the data.

The results of the experiment\footnote{All experimental results in this paper were obtained on a laptop with the Intel Core i7-8650U CPU (1.90GHz x 8) and 16 GB DDR4 RAM, no parallelism was used. The source code is available at \href{https://github.com/arodomanov/rcdvs}{https://github.com/arodomanov/rcdvs}.} for different values of $n$ and the eigenvalue ratio $\lambda_1/\lambda_2$ are shown in Figure~\ref{figure:experiment_results_quadratic_function}. Each column in this table displays the median value of the corresponding statistic: ``It'' is the total number of iterations (in thousands) taken by the method until its termination; ``T'' is the corresponding total running time (in seconds); ``Acc'' is the actual acceleration ratio of the method over RCD in terms of the number of iterations (this ratio may be less than 1 when there is no acceleration); ``\%'' expresses the ``Acc'' for RCDVS as a percent of the theoretical prediction \eqref{eq:acceleration_ratio_over_RCD}.\footnote{To keep the table concise, we report only the first few significant digits for each statistic. For example, in the first row of the table, the value 2 in the column ``It'' for RCDVS may actually stand for 2.1, 2.53, or even 2.999.} From this table, we can see that the number of iterations for RCD and SDNA grows significantly with the spectral gap between the two largest eigenvalues, while for RCDVS there is almost no growth at all. As a result, RCDVS dramatically outperforms the other two methods both in terms of iterations and total running time, especially for large values of $\lambda_1 / \lambda_2$. By inspecting the ``Acc'' column, we observe that the actual acceleration ratio of RCDVS with respect to RCD monotonically increases with the spectral gap, which is natural. What is more important, the ``\%'' always takes values around 100, which means that our theoretical prediction \eqref{eq:acceleration_ratio_over_RCD} is quite accurate. SDNA, on the contrary, performs even worse than RCD in most cases.

\begin{figure}\scriptsize
\begin{center}
\begin{tabular}{|r|r||r|r||r|r|r||r|r|r|r|} \hline
\multicolumn{2}{|c||}{Parameters} & \multicolumn{2}{c||}{RCD} & \multicolumn{3}{c||}{SDNA} & \multicolumn{4}{c|}{RCDVS} \\ \hline
$n$   & $\lambda_1/\lambda_2$ & It        & T        & It       & Acc    & T      & It     & Acc    & \%         & T  \\ \hline
400   &   4                   &    5      &  0.0     &   15     &   0.3  &  0.1   &    2   &   2    &  118       & 0.05 \\
      &  16                   &   12      &  0.1     &   44     &   0.4  &  0.2   &    2   &   4    &  105       & 0.05 \\
      &  64                   &   37      &  0.2     &   92     &   0.4  &  0.4   &    3   &  11    &   83       & 0.05 \\
      & 256                   &  126      &  0.4     &  125     &   1.2  &  0.5   &    3   &  40    &   77       & 0.05 \\
      & 1,024                 &  499      &  1.3     &  137     &   3.8  &  0.6   &    3   & 132    &   64       & 0.06 \\ \hline

800   &   4                   &    8      &  0.1     &    7     &  1.0  &  0.1    &   4    &   2    &  148       &  0.19 \\
      &  16                   &   18      &  0.2     &   50     &  0.4  &  0.4    &   5    &   3    &  140       &  0.20 \\
      &  64                   &   47      &  0.3     &  156     &  0.3  &  1.1    &   5    &   9    &  115       &  0.20 \\
      & 256                   &  163      &  0.7     &  356     &  0.5  &  2.8    &   6    &  27    &   91       & 0.21 \\
      & 1,024                 &  576      &  2.5     &  526     &  1.1  &  4.3    &   6    &  97    &   84       & 0.21 \\ \hline

1,600 &   4                  &   18      &  0.5     &    22    &  0.8 &   1.0    &   7    &   2    &  189       &  1.03 \\
      &   16                 &   24      &  0.5     &    69    &  0.4 &   3.0    &   9    &   2    &  134       &  1.05 \\
      &   64                 &   65      &  0.9     &   237    &  0.3 &  10.1    &  10    &   6    &  125       &  1.06 \\
      &  256                 &  202      &  2.1     &   794    &  0.3 &  35.0    &  13    &  14    &   87       &  1.17 \\
      & 1,024                &  758      &  6.7     &  2,034   &  0.3 &  86.6    &  14    &  48    &   79       &  1.20 \\ \hline

3,200 &    4                 &   28      &  1.7     &    39    &  0.8 &   3.7    &  15    &   1    &  167       & 4.14 \\
      &   16                 &   41      &  2.1     &   107    &  0.4 &   9.7    &  18    &   2    &  151       & 4.32 \\
      &   64                 &   81      &  3.0     &   428    &  0.2 &  39.4    &  24    &   3    &  116       & 4.75 \\
      &  256                 &  228      &  6.0     &  1,589   &  0.1 & 144.7    &  23    &   9    &  113       & 4.74 \\
      & 1,024                &  828      & 17.1     &  5,026   &  0.2 &  468.0   &   26   &   31   &    97      &  4.91 \\ \hline

\end{tabular}
\end{center}
\caption{\label{figure:experiment_results_quadratic_function}Results for the quadratic function.}
\end{figure}

\subsection{Huber function}

In the second experiment, we still use $\tau=2$ but now consider the Huber function from Example~\ref{example:huber_function}. In contrast to the previous one, this objective is non-strongly convex.

The design of the experiment is almost the same as before. To generate the matrix $A$, we choose $\lambda_1 \geq \lambda_2 := 100$, set $A$ to be the $m \times n$ diagonal matrix with elements $\sqrt{\lambda_1 / \mu}, \sqrt{\lambda_2 / \mu}, \sqrt{1 / \mu}, \dots, \sqrt{1 / \mu}$, where $\mu := 0.01$, and then successively perform 10 random Householder reflections $A \mapsto (I - 2 u u^T) A (I - 2 v v^T)$ on the rows and columns of $A$, where each time $u$ and $v$ are uniformly distributed on the unit spheres in $\R^m$ and $\R^n$ respectively. Such a construction ensures that the matrix $B := \frac{1}{\mu} A^T A$ (see Section~\ref{section:smoothing_technique}) has eigenvalues $\lambda_1, \lambda_2, 1, \dots, 1$ (plus $n - m$ zeros when $m < n$). The vector $b$, the starting point $x_0$ and the termination criterion for the methods are absolutely the same as in the previous experiment.

One additional remark should be made in the case when $m < n$. Recall that in this situation $B$ is in fact degenerate and hence some of its principal submatrices may not be invertible. This does not cause any difficulties for RCD and RCDVS since the probability of choosing a degenerate submatrix in these methods is zero. However, this is not so for SDNA and thus, strictly speaking, SDNA is not well-defined in the case of a degenerate matrix $B$. To fix this problem, we use the More--Penrose pseudoinverse instead of the usual inverse in this method. (We have also tried to simply skip the update when a degenerate submatrix has been chosen, but this strategy turned out to work somewhat worse.)

The results of the experiment are shown in Figure~\ref{figure:experiment_results_dense_huber_function}. In the same way as before, we see that RCDVS always significantly outperforms RCD and its actual acceleration ratio is quite close to the theoretical prediction. However, it is interesting that this time SDNA works quite well for problems with $m < n$ (almost comparably to RCDVS). In particular, its number of iterations is almost independent of the spectral gap. Nevertheless, for $n > m$ its behavior is the same as in the previous experiment.

Now let us consider the same problem but with much bigger dimensions. For this, we slightly change the way we construct $A$. This time, we generate it as a \emph{sparse} matrix using the procedure described above but with sparse directions $u$ and $v$ that are chosen as follows. First, we take some integer $1 \leq p \leq \min\{m, n\}$ that controls the resulting sparsity level of $B$. After this, we pick $p$ random uniformly distributed indices and fill the positions corresponding to these indices with a random vector uniformly distributed on the unit sphere in $\R^p$, the rest of the positions are set to zero. Of course, in order to work with a sparse problem, every method has to be properly modified. In particular, we should use the special algorithm from Section~\ref{section:sparse_two_element_volume_sampling} for doing volume sampling in RCDVS.

The results for the bigger dimensions are shown in Figure~\ref{figure:experiment_results_sparse_huber_function}. Somewhat surprisingly, the problems with $m < n$ are now even more difficult for SDNA than those with $m > n$. Otherwise, the overall picture is the same as previously.

\begin{figure}\scriptsize
\begin{center}
\begin{tabular}{|r|r|r||r|r||r|r|r||r|r|r|r|} \hline
\multicolumn{3}{|c||}{Parameters} & \multicolumn{2}{c||}{RCD} & \multicolumn{3}{c||}{SDNA} & \multicolumn{4}{c|}{RCDVS} \\ \hline
$m$      & $n$      & $\lambda_1/\lambda_2$ & It        & T       & It       & Acc   & T      & It     & Acc    & \%      & T    \\ \hline
   400   &  800     &    4                  &  14       &  0.3    &    7     &   1.9 &  0.2   &   6    &   2    & 131     &  0.3 \\
         &          &   16                  &   34      &  0.5    &    8     &   3.9 &  0.2   &   7    &   4    & 111     &  0.3 \\
         &          &   64                  &  114      &  1.6    &   10     &  11.4 &  0.2   &   8    &   13   &  99     &  0.3 \\
         &          &  256                  &   409     &  6.3    &   12     &  34.3 &  0.3   &   8    &   49   &  94     &  0.3 \\
         &          & 1,024                 & 1,529     & 21.5    &   14     & 118.7 &  0.3   &   8    &  168   &  81     &  0.3 \\ \hline

   800   &  400     &   4                   &  13       &  0.3    &  20      &  0.7  &  0.6   &  5     &    2   & 147     &  0.2 \\
         &          &  16                   &  31       &  0.7    &   57     &   0.6 &   1.4  &   6    &     4  &  107    &  0.2 \\
         &          &   64                  &  103      &  2.0    &  118     &   0.9 &   2.8  &    7   &    13  &   96    &  0.2 \\
         &          & 256                   &  337      &  6.5    &  165     &   2.2 &  4.3   &    7   &    44  &   84    &  0.2 \\
         &          & 1,024                 & 1,402     &  27.3   &  184     &   7.5 &  4.8   &    7   &   172  &   83    &  0.2 \\ \hline

   800   & 1,600    &    4                  &   23      &  0.5    &   23     &  0.9  &  0.9   &   11   &   2    &  143    &  1.0 \\
         &          &   16                  &   45      &  1.0    &   22     &  2.1  &   0.9  &   13   &   3    &  122    &  1.1 \\
         &          &   64                  &  122      &  2.9    &   23     &  5.9  &  1.0   &   14   &   8    &  101    &  1.3 \\
         &          &  256                  &  456      &  9.6    &   23     & 18.8  &  1.0   &   15   &  28    &   98    &  1.2 \\
         &          &  1,024                & 1,646     & 33.8    &   28     & 57.3  &  1.1   &   16   & 104    &   91    &  1.2 \\ \hline

  1,600  &  800     &    4                  &   22      &  0.9    &   25     &  0.8  &   1.7  &   10   &    2   &  155    &  0.8 \\
         &          &  16                   &   44      &  1.6    &   66     &  0.7  &   4.6  &   13   &   3    &  120    &  0.9 \\
         &          &  64                   &  119      &  4.2    &  184     &  0.6  & 11.7   &   14   &   8    &  101    &  1.0 \\
         &          &  256                  &  438      & 15.2    &  406     &  1.0  & 26.4   &   14   &  29    &   98    &  1.0 \\
         &          &  1,024                & 1,588     & 55.1    &  625     &  2.6  & 40.9   &   15   &  98    &   85    &  1.0 \\ \hline
\end{tabular}
\end{center}
\caption{\label{figure:experiment_results_dense_huber_function}Results for the Huber function.}
\end{figure}

\begin{figure}\scriptsize
\begin{center}
\begin{tabular}{|r|r|r||r|r||r|r|r||r|r|r|r|} \hline
\multicolumn{3}{|c||}{Parameters} & \multicolumn{2}{c||}{RCD} & \multicolumn{3}{c||}{SDNA} & \multicolumn{4}{c|}{RCDVS}           \\ \hline
$m$      & $n$      & $\lambda_1/\lambda_2$    & It        & T     & It       & Acc   & T      & It   & Acc    & \%      & T      \\ \hline
8,000    &   16,000 &    64                    &   352     &   3.8 &   2,981  &   0.1 &   36.9 & 125  &     2  &   153   &   1.8 \\
         &          &   256                    &   710     &   7.3 &   6,326  &   0.1 &   77.3 & 151  &     4  &   111   &   2.1 \\
         &          &  1,024                   &  2,328    &  24.5 &  12,250  &   0.2 &  154.0 & 169  &    13  &   101   &   2.3 \\
         &          &  4,096                   &  8,855    &  88.3 &  24,742  &   0.3 &  304.2 & 176  &   49   &    95   &   2.3 \\
         &          &  16,384                  &  35,917   &  363.9&  50,049  &   0.7 &  624.9 & 180  &  200   &    98   &   2.5 \\ \hline

16,000   &    8,000 &    64                    &    332    &   3.6 &   1,351  &   0.2 &   19.3 & 136  &    2   &   134   &   1.9 \\
         &          &   256                    &    727    &   7.7 &   2,654  &   0.3 &   37.9 & 166  &    4   &   107   &   2.2 \\
         &          &  1,024                   &   2,306   &  23.3 &   5,364  &   0.5 &   76.0 & 170  &    13  &   102   &   2.3 \\
         &          &  4,096                   &   8,648   &  87.1 &  10,682  &   0.9 &  154.2 & 175  &    50  &    98   &  2.3 \\
         &          & 16,384                   &  37,623   &  378.3&  21,313  &   1.9 &  303.2 & 181  &   214  &   105   &  2.4 \\ \hline

16,000   &   32,000 &    64                    &   518     &   5.6 &    5,829 &   0.1 &  73.1  & 232  &     2  &    158  &  3.7 \\
         &          &   256                    &   944     &   9.7 &   14,751 &   0.1 &  184.1 & 282  &     3  &    128  &   4.2 \\
         &          &  1,024                   &  2,642    &  27.7 &   32,066 &   0.1 &  410.1 & 352  &    7   &    105  &    5.0 \\
         &          &  4,096                   &  9,539    &  97.2 &   64,144 &   0.1 &  822.1 & 358  &   26   &     98  &   5.3 \\
         &          & 16,384                   &  36,621   &  367.0&  110,638 &   0.3 & 1,402.8& 367  &   98   &    95   &   5.3 \\ \hline

32,000   &  16,000  &    64                   &   501     &    5.7&   2,077  &   0.2 &   28.7 & 228  &    2   &   156   &   3.7 \\
         &          &   256                   &   918     &   10.3&   4,220  &   0.2 &   58.6 & 301  &     3  &   119   &   4.5 \\
         &          &  1,024                  &  2,683    &   27.4&   8,610  &   0.3 &  122.3 & 352  &     7  &   101   &   5.0 \\
         &          &  4,096                  &  9,790    &   99.4&  16,909  &   0.6 &  239.1 & 364  &    27  &   103   &   5.1 \\
         &          &  16,384                  & 38,055    &  384.1&  34,351  &   1.1 &  487.4 & 375  &    102 &    99   &   5.4 \\ \hline
\end{tabular}
\end{center}
\caption{\label{figure:experiment_results_sparse_huber_function}Results for the Huber function (sparse matrix). For dimensions $8,000 \times 16,000$ and $16,000 \times 8,000$ we use $p := 50$, while for $16,000 \times 32,000$ and $32,000 \times 16,000$ we use $p := 70$.}
\end{figure}

\subsection{Logistic regression}

Now we consider the $l^2$-regularized logistic regression function from Example~\ref{example:l2_regularized_logistic_regression}, which is very popular in the context of machine learning. The termination criterion for each method is the same as before with the difference that now the optimal objective value is unknown and we have to calculate it numerically in advance. Nevertheless, this auxiliary computation is needed only for our presentation and does not affect the actual performance of the methods in any way.

We set $\gamma := 1$ since this is a default value of the regularization parameter used in practice. However, instead of generating the data $a_1, \dots, a_m$ and $b_1, \dots, b_m$ artificially as we did before, now we take some real-world data from the LIBSVM website\footnote{\url{https://www.csie.ntu.edu.tw/~cjlin/libsvmtools/datasets/binary.html}}, which is summarized in Figure~\ref{figure:experiment_datasets}. Here $m$ is the number of observations and $n$ is the number of features. The next 4 columns display the four largest eigenvalues of the matrix $B := \frac{1}{4} \sum_{i=1}^m a_i a_i^T + \gamma I$ (see Example~\ref{example:l2_regularized_logistic_regression}), while the last 3 columns show the theoretical acceleration ratio \eqref{eq:acceleration_ratio_over_RCD} for the three corresponding values of $\tau$. The main reason why we are presenting this table is to demonstrate that it is not uncommon for real data to have significant spectral gaps between the first largest eigenvalues (although these gaps are not as big as in our previous experiments with artificial data).

For the results of the experiment, see Figure~\ref{figure:experiment_results_logistic_regression}, where, in contrast to the previous two experiments, we additionally consider several small values of $\tau$ for RCDVS and SDNA. We can see that, on the real data, the method SDNA looks much better than previously and in many cases it outperforms RCD. Nevertheless, RCDVS is still a winner, and its actual acceleration rate is usually even faster than predicted by theory.

\begin{figure}\scriptsize
\begin{center}
\begin{tabular}{|l|r|r||r|r|r|r||r|r|r|} \hline
\multirow{2}{*}{Data} & \multirow{2}{*}{$m$} & \multirow{2}{*}{$n$} & \multicolumn{4}{c||}{\multirow{2}{*}{Top 4 eigenvalues}} & \multicolumn{3}{c|}{TheoryAcc} \\ \cline{8-10}
&                         &        & \multicolumn{4}{c||}{}                       & 2      & 3      & 4 \\ \hline
breast-cancer             & 683    &  10   &  891     &  118    &  41    &  35    & 4.0    & 6.6    & 8.6 \\ \hline
phishing                  & 11,055 &  68   &  1,798   &  109    &  102   &  95    & 2.7    & 3.1    & 3.4 \\ \hline
a9a                       & 32,561 &  123  &  51,184  &  7,502  &  4,745 &  3,695 & 1.8    & 2.1    & 2.3 \\ \hline
\end{tabular}
\end{center}
\caption{\label{figure:experiment_datasets}Real data for logistic regression.}
\end{figure}

\begin{figure}\scriptsize
\begin{center}
\begin{tabular}{|l|r||r|r||r|r|r||r|r|r|r|} \hline
\multicolumn{2}{|c||}{Parameters} & \multicolumn{2}{c||}{RCD} & \multicolumn{3}{c||}{SDNA} & \multicolumn{4}{c|}{RCDVS} \\ \hline
Data            & $\tau$     & It           & T         & It           & Acc   & T       & It         & Acc    & \%      & T      \\ \hline
breast-cancer   & 2          &  1.8         & 0.10      &  0.6         &   3.2 & 0.03    &   0.4      &    4   &  101    & 0.02 \\
                & 3          &              &           &   0.3        &   5.4 & 0.02    &   0.3      &    6   &   96    & 0.02 \\
                & 4          &              &           &   0.2        &   7.2 & 0.02    &   0.1      &   12   &  148    &  0.01 \\ \hline

phishing        & 2          & 11.0         & 7.12      &   4.8        &   2.3 & 3.74    &  3.3       &   3    &   120   &  2.80 \\
                & 3          &              &           &  2.7         &   4.0 & 2.36    &  2.1       &   5    &  169    &  1.69 \\
                & 4          &              &           &  1.8         &   6.3 & 1.93    &  1.6       &   6    &  202    &  2.26 \\ \hline

a9a             & 2          & 290.9        & 949.40    & 444.1        &   0.7 & 1,583   & 175.6      &   1    &   90    & 613.44 \\
                & 3          &              &           & 273.3        &  1.1  & 1,161   & 73.8       &   3    &  189    & 309.02 \\
                & 4          &              &           & 171.8        &  1.7  & 839     &  43.0      &   6    &  296    & 212.13 \\ \hline
\end{tabular}
\end{center}
\caption{\label{figure:experiment_results_logistic_regression}Results for logistic regression.}
\end{figure}

\section{Conclusion}

We have analyzed the randomized coordinate descent method with volume sampling (RCDVS) for minimizing a function, that is smooth with respect to some positive semidefinite matrix $B$. In its iterations, this method randomly selects $\tau$-element subsets of coordinates with probabilities proportional to the determinants of principal submatrices of $B$. We have shown, both theoretically and empirically, that the increase in $\tau$ from $\tau_1$ to $\tau_2$ improves the convergence rate up to the factor~\eqref{eq:theoretical_acceleration_ratio}, which depends on the spectral gap between the $\tau_1$st and $\tau_2$nd eigenvalues of $B$.

However, there are still many important directions for further research:
\begin{itemize}
\item \textbf{Accelerated method.} In addition to the basic randomized coordinate descent, there also exists the \emph{accelerated} one \cite{nesterov2017efficiency,allen2016even}, where the coordinates are sampled with probabilities proportional to the \emph{square roots} of the diagonal elements of $B$. Is it possible to accelerate RCDVS in a similar manner, possibly using the square roots of the determinants as probabilities?

\item \textbf{Constrained and composite optimization.} We have considered only the basic smooth \emph{unconstrained} minimization. However, most optimization methods can often be extended to handle problems involving some \emph{simple constraints} (e.g. box constraints or linear ones), or they can also be extended to working with \emph{composite functions} (when the objective is the sum of a smooth function and a simple convex possibly non-smooth function), while still retaining the original convergence rate. Can we generalize RCDVS to these settings?\footnote{Currently, the main problem here is that, up to our knowledge, there are no corresponding results even for the case $\tau=1$, that is for the RCD method.  }

\item \textbf{Special volume sampling algorithms or different kind of sampling.} Although the results, that we have obtained, are true for \emph{any} value of $\tau$, from the \emph{practical} point of view, currently there is only one choice that is suitable for large scale problems, namely $\tau=2$ (apart from previously known $\tau=1$). The problem is that currently there are no algorithms for volume sampling whose preprocessing/sampling complexity is appropriate for large scale applications (e.g. $O(n^2)$ instead of $O(n^3)$). However, this does not mean that it is not possible to devise such algorithms, especially when the matrix possesses special \emph{structure} (e.g. sparse, banded, low-rank etc.). Another interesting question is whether volume sampling can be replaced with some other kind of sampling which is more practical but still gives similar results.
\end{itemize}

\section*{Acknowledgments}
The authors are thankful to the anonymous reviewers for
their attentive reading and valuable comments.

\pagebreak

\bibliographystyle{unsrt}
\bibliography{references}

\appendix

\section{Proof of Theorem~\ref{theorem:convergence_for_convex_functions}}
\label{sec-proof-conv}

\begin{proof}
Note that $\Argmin f$ is non-empty and compact by the Weierstrass theorem ($f$ is continuous as a convex function with open domain, the sublevel set $L_f(x_0)$ is bounded by the statement and is closed as the inverse image of a closed set under a continuous mapping). Hence, both $\min$ and $\max$ in the definition of $D_\tau$ are attained and, in particular, $D_\tau$ is finite.

Using Lemma~\ref{lemma:sufficient_decrease_lemma_2}, we obtain $x_k \in L_f(x_0)$ and
\begin{equation}\label{eq:convergence_for_convex_functions:proof:1}
\E f(x_k) - \E f(x_{k+1}) \geq \frac{1}{2} \E \| \nabla f(x_k) \|^2_{(B_{\tau})^{-1}}
\end{equation}
for all $k \geq 0$.

Let $k \geq 0$. By the convexity of $f$ and Cauchy-Schwarz inequality, we have
$$
f(x_k) - \min f \leq \langle \nabla f(x_k), x_k - x^* \rangle \leq \| \nabla f(x_k) \|_{(B_\tau)^{-1}} \| x_k - x^* \|_{B_\tau}
$$
Hence, by the definition of $D_{\tau}$, it follows that
$$
f(x_k) - \min f \leq \| \nabla f(x_k) \|_{(B_\tau)^{-1}} D_{\tau},
$$
from which, by Jensen's inequality, we conclude that
\begin{equation}\label{eq:convergence_for_convex_functions:proof:2}
\E \| \nabla f(x_k) \|^2_{(B_{\tau})^{-1}} \geq \frac{ \E (f(x_k) - \min f)^2 }{D_\tau^2} \geq \frac{ (\E f(x_k) - \min f)^2 }{D_\tau^2}.
\end{equation}

Combining \eqref{eq:convergence_for_convex_functions:proof:1} and \eqref{eq:convergence_for_convex_functions:proof:2} and writing $\delta_k := \E f(x_k) - \min f$, we finally obtain
\begin{equation}\label{eq:convergence_for_convex_functions:proof:3}
\delta_k - \delta_{k+1} \geq \frac{\delta_k^2}{2 D_\tau^2}
\end{equation}
for all $k \geq 0$. Now the claim follows by a standard argument. Indeed, we can assume without loss of generality that $\delta_k$ is strictly positive for each $k \geq 0$. Then, using \eqref{eq:convergence_for_convex_functions:proof:3} together with the monotonicity of $\delta_k$, we obtain
$$
\frac{1}{\delta_{k+1}} - \frac{1}{\delta_k} = \frac{\delta_k - \delta_{k+1}}{\delta_k \delta_{k+1}} \geq \frac{1}{2 D_\tau^2}
$$
for all $k \geq 0$. By induction, it follows that
$$
\frac{1}{\delta_k} \geq \frac{1}{\delta_0} + \frac{k}{2 D_\tau^2} \geq \frac{k + 1}{2 D_\tau^2}
$$
for all $k \geq 0$, where the last inequality is a consequence of \eqref{eq:convergence_for_convex_functions:proof:3} and the positivity of $\delta_1$.
\end{proof}

\section{Proof of Theorem~\ref{theorem:convergence_for_strongly_convex_functions}}
\label{sec-proof-sconv}

\begin{proof}
Let $0 \leq k \leq K$. By the same argument as in the proof of Theorem~\ref{theorem:convergence_for_convex_functions},
\begin{equation}\label{st-conv-1}
\E f(x_k) - \E f(x_{k+1}) \geq \frac{1}{2} \E \| \nabla f(x_k) \|_{(B_\tau)^{-1}}^2.
\end{equation}
At the same time, by strong convexity of $f$ in the norm $\| \cdot \|_{B_\tau}$, we have\footnote{This is a standard inequality for strongly convex functions, and can be easily proved from the definition~\eqref{eq:strong_convexity} by minimizing both sides in $y \in \R^n$.}
\begin{equation}\label{st-conv-2}
\frac{1}{2 \mu_\tau} \| \nabla f(x_k) \|_{(B_\tau)^{-1}}^2 \geq f(x_k) - \min f.
\end{equation}
Combining \eqref{st-conv-1} and \eqref{st-conv-2}, we obtain that
$$
\E f(x_k) - \E f(x_{k+1}) \geq \mu_\tau (\E f(x_k) - \min f).
$$
or, equivalently, after rearranging, that
$$
\E f(x_{k+1}) \leq \E f(x_k) - \mu_\tau (\E f(x_k) - \min f).
$$
Thus, subtracting $\min f$ from both sides, we get
$$
\E f(x_{k+1}) - \min f \leq (1 - \mu_\tau) (\E f(x_k) - \min f).
$$
The claim now follows by induction.
\end{proof}

\end{document}